\documentclass[11pt]{amsart}
\pdfoutput=1
\usepackage[utf8]{inputenc}
\usepackage[T1]{fontenc}
\usepackage{lmodern}

\usepackage{amssymb}
\usepackage{mathtools}
\usepackage{amscd}
\usepackage[mathscr]{eucal}
\usepackage{enumitem}
\usepackage[hmargin=1in,vmargin=1in]{geometry}

\usepackage{graphicx}
\usepackage{xcolor}

\usepackage[roman]{sublabel}

\usepackage[labelfont=bf,font=small]{caption}
\usepackage{subcaption}

\usepackage[noadjust]{cite}
\usepackage[foot]{amsaddr}

\usepackage{ottmacros-Geiger-arxiv-v5}

\DeclareMathOperator{\LY}{LY}
\DeclareMathOperator{\Lip}{Lip}

\DeclareMathOperator{\osc}{osc}
\DeclareMathOperator{\OSC}{OSC}
\DeclareMathOperator{\Hausd}{Hausd}
\DeclareMathOperator{\Einf}{Einf}
\DeclareMathOperator{\Esup}{Esup}

\definecolor{cite_blue}{rgb}{0,0,1}
\definecolor{link_red}{rgb}{0.75,0,0}
\usepackage[colorlinks=true,citecolor=cite_blue,linkcolor=link_red]{hyperref}

\theoremstyle{plain}
\newtheorem{theorem}{Theorem}[section]
\newtheorem{corollary}[theorem]{Corollary}
\newtheorem{lemma}[theorem]{Lemma}
\newtheorem{proposition}[theorem]{Proposition}

\newtheorem{MainTheorem}{Primary Theorem}

\theoremstyle{definition}
\newtheorem{definition}[theorem]{Definition}
\newtheorem{notation}[theorem]{Notation}
\newtheorem{remark}[theorem]{Remark}

\newtheorem{direction}[theorem]{Future Direction}

\newtheorem*{acknowledgments}{Acknowledgments}

\begin{document}

\title{Nonstationary open dynamical systems}

\author{Brett Geiger$^{1}$}
\address{$^{1}$Southern Methodist University, Dallas, Texas, USA}

\author{William Ott$^{2}$}
\address{$^{2}$University of Houston, Houston, Texas, USA}

\email{bgeiger@mail.smu.edu, ott@math.uh.edu}




\date{\today}

\begin{abstract}
This paper studies nonstationary open dynamical systems from the statistical viewpoint.
By open, we mean that trajectories may escape through holes in the phase space.
By nonstationary, we mean that the dynamical model itself (as well as the holes) may vary with time.
Unlike the setting of random dynamical systems, no assumptions are made about the statistics of this variability.
We formulate general results that yield conditional memory loss (an analog of decay of correlations) at exponential rates for nonstationary open dynamical systems.
We apply our results to a class of nonstationary piecewise-smooth expanding systems in higher dimension with moving holes.
\end{abstract}

\maketitle

\thispagestyle{empty}

\raggedbottom

\section{Introduction}
\label{s:intro}

Physical processes in the natural world often evolve in temporally varying environments.
Examples range from genetic regulatory networks operating {\itshape in vivo}~\cite{Blanchard-2018, Veliz-Cuba-2016} to collisional systems modeled by hyperbolic billiards with moving scatterers~\cite{Stenlund-LSY-Zhang-2013}.
In response, ergodic-theoretic techniques have been brought to bear on two classes of nonequilibrium dynamical systems: nonstationary dynamics and open systems.
By {\itshape nonstationary}, we mean that the dynamical model itself evolves over time.
By {\itshape open}, we refer to systems with holes through which trajectories may escape.
We aim here to unify these two lines of inquiry.
This paper develops an analytical framework to treat systems that are both nonstationary and open.
That is, we allow the dynamical model itself to evolve over time and we allow holes that can move and deform.
Importantly, unlike the setting of random dynamical systems~\cite{Arnold-RDS-1998}, we do not assume knowledge of the statistics of the temporal variability of the model and the holes.

When analyzing nonstationary systems, one should think not about invariant measures and their statistical properties, but rather about the co-evolution of pairs of trajectories or pairs of initial measures.
For a predominantly contractive nonstationary system, it is reasonable to expect that a single common future will be shared by almost every pair of trajectories.
Consider for example a sequence of contractions $(\hat{f}_{i})_{i=1}^{\infty}$ on a compact metric space $(X,d)$:
There exists $0 < L < 1$ such that $d(\hat{f}_{i} (p), \hat{f}_{i} (q)) \leqs L d(p,q)$ for all $i \in \mbb{N}$ and $p,q \in X$.
Form a nonstationary system by applying the maps $\wh{F}_{m} = \hat{f}_{m} \circ \hat{f}_{m-1} \circ \cdots \circ \hat{f}_{1}$ to $X$.
If all of the $\hat{f}_{i}$ are the same, the contraction mapping principle yields a unique fixed point to which every trajectory will converge.
The general picture, however, is this:
Individual trajectories may wander forever, but a single future is shared by every pair of trajectories since $\diam (\wh{F}_{m} (X)) \to 0$ as $m \to \infty$.
Put another way, the metric space coalesces into a small blob that continues to shrink and move as time progresses.

Stochastic differential equations with contractive properties exhibit an analogous phenomenon.
Kunita~\cite{Kunita-flows-1990} has shown that a stochastic differential equation of the form
\begin{equation*}
\mrm{d} x_{t} = a (x_{t}) \mrm{d} t + \sum_{i=1}^{n} b_{i} (x_{t}) \circ \mrm{d} W_{t}^{i}
\end{equation*}
generates a stochastic flow of diffeomorphisms, wherein almost every Brownian path defines a time-dependent flow.
Lyapunov exponents for such flows are known to be well-defined, {\itshape nonrandom} (they do not depend on noise realization), and constant almost everywhere in phase space if the system is ergodic.
For ergodic systems, the sign of the greatest Lyapunov exponent $\Lambda_{\max}$ conveys a great deal of information about the nature of the random dynamics.
If $\Lambda_{\max} < 0$, then mild additional assumptions imply that any ensemble of initial conditions will coalesce near a unique equilibrium point that evolves in time~\cite{LeJan-1985}.
This phenomenon, colloquially known as {\itshape random sinks}, occurs for dissipative systems such as the stochastically-forced two-dimensional Navier-Stokes equations~\cite{Masmoudi-Young-2002, Mattingly-1999}.
When $\Lambda_{\max} > 0$, initial distributions will track intricately-structured random Sinai-Ruelle-Bowen (SRB) measures~\cite{Ledrappier-Young-1988}.
This result exemplifies general ideas in dynamics:
Local instability leads to global statistical regularity.
Systems with some hyperbolicity induce contractive behavior, not in phase space as before, but in suitable spaces of measures.

For autonomous deterministic systems, the notion of statistical return to equilibrium exemplifies this type of contractive behavior: Suitable absolutely continuous probability measures converge to the unique invariant measure when evolved by the dynamics.
Return to equilibrium and the related notion of decay of correlations have received extensive attention in the autonomous deterministic context.
Techniques used to analyze correlation decay rates include spectral transfer operator analysis ({\itshape e.g.}~\cite{Baladi-Demers-Liverani-2018, Baladi-Tsujii-2007, Demers-Zhang-2011, gouezelliverani2006, Young-towers-1998}), coupling ({\itshape e.g.}~\cite{Bressaud-Liverani-2002, Chernov-coupling-2006, Young-coupling-1999}), convex cones and the projective Hilbert metric ({\itshape e.g.}~\cite{Liverani-cones-1995}), and operator renewal theory ({\itshape e.g.}~\cite{Melbourne-Terhesiu-2012}).

For nonstationary systems, it is natural to consider not return to equilibrium, but rather decay of the distance between two measures as they evolve.
We say that such a system exhibits memory loss in the statistical sense if for every suitable pair $\varphi_{0}$ and $\psi_{0}$ of initial probability densities, we have $\lim_{t \to \infty} \norm{\varphi_{t} - \psi_{t}} = 0$, where $\varphi_{t}$ and $\psi_{t}$ denote the evolved probability densities.
Both the strength of the norm and the rate of convergence are of interest.
Exponential-rate statistical memory loss has been established for smooth expanding maps and one-dimensional piecewise-smooth expanding maps~\cite{Ott-Stenlund-Young-2009}, nonstationary Anosov systems~\cite{Stenlund-Anosov-2011}, piecewise-smooth expanding systems in higher dimension~\cite{GuptaOttTorok2013}, and dispersing billiards with moving scatterers~\cite{Stenlund-LSY-Zhang-2013}.
The billiard result raises an interesting question: How do scatterer movement (and perhaps deformation) emerge from the microscopic physics?
See~\cite{Chernov-Dolgopyat-2009} for an effort to model scatterer movement under bombardment by light particles.
Aimino {\itshape et al.}~\cite{Aimino-polynomial-2015} establish polynomial-rate statistical memory loss for nonstationary one-dimensional intermittent systems of Pomeau-Manneville type.

The study of autonomous dynamical systems with holes (open systems) originates in no small part with the work of Pianigiani and Yorke~\cite{Pianigiani-Yorke-1979}.
Many important questions may be asked about such systems.
Does mass escape at a well-defined asymptotic rate?
Where should holes be placed so as to maximize escape rates~\cite{Bunimovich-Yurchenko-2011}?
How do the conditionally invariant measures organize system statistics?
What is the nature of the dynamics on the set of points that never enter a hole (the survivor set)?
See~\cite{Demers-Young-open-2006} for a semi-expository point of entry into the open systems literature.
Specific models analyzed thus far include the logistic family~\cite{Demers-logistic-2005}, Lasota-Yorke maps~\cite{Liverani_Deschamps_2003}, dispersing billiards in dimension two~\cite{Demers-billiards-2014, Demers-Wright-LSY-2010}, and spherical billiards~\cite{Dettmann-Rahman-2014}.

Here, we study conditional memory loss, an analog of decay of correlations for systems that are both nonstationary and open.
Loosely speaking, a nonstationary open system exhibits conditional memory loss in the statistical sense if for every pair $\varphi_{0}$ and $\psi_{0}$ of suitable initial probability densities, we have $\norm{ (\varphi_{t} / \| \varphi_{t} \|) - (\psi_{t} / \| \psi_{t} \|) } \to 0$ as $t \to \infty$, where $\norm{\cdot}$ denotes the $L^{1}$ norm with respect to the reference measure.
That is, the distance between the two measures decays if we normalize by the mass that has yet to escape.
We formulate a general framework for establishing exponential-rate conditional memory loss in the statistical sense.
We then demonstrate the utility of our approach by applying our two primary theorems to a class of piecewise-smooth expanding systems in higher dimension studied by Saussol~\cite{SaussolQuasi2000} and by Gupta, Ott, and T\"{o}r\"{o}k~\cite{GuptaOttTorok2013}.

\section{A general framework for conditional memory loss}
\label{s:framework1}

Let $X$ be a compact Riemannian manifold and let $\lambda $ denote Riemannian volume on $X$.  Assume $\lambda(X) = 1.$
We begin by defining nonstationary dynamical systems on $X$.
Consider a sequence $(\hat{f}_{i} : X \to X)_{i=1}^{\infty }$.
For $m \in \mbb{N}$, define $\wh{F}_{m} = \hat{f}_{m} \circ \hat{f}_{m-1} \circ \cdots \circ \hat{f}_{1}$.
We call the sequence $(\wh{F}_{m})_{m=1}^{\infty }$ a {\bfseries \itshape nonstationary dynamical system}.
An open system is created from $(\wh{F}_{m})_{m=1}^{\infty }$ by introducing holes.
For $j \in \mbb{N}$, let $H_{j} \subset X$.
We call $H_{j}$ the {\itshape hole} at time $j$.
For $m \in \mbb{N}$, define the time-$m$ {\itshape survivor set} $S_{m}$ by
\begin{equation*}
S_{m} = X \setminus \bigcup_{i=1}^{m} (\wh{F}_{i}^{-1} (H_{i})).
\end{equation*}
Let $F_{m}$ denote the restriction $\wh{F}_{m} | S_{m}$; that is, $F_{m}$ is defined on points with trajectories that have not fallen into a hole after $m$ iterates.
We call the pair $((F_{m}), (H_{j}))$ a {\bfseries \itshape nonstationary open dynamical system}.

In Sections~\ref{ssection:set1} and \ref{ss:primary-theorems}, we formulate a general framework that allows one to establish a type of statistical memory loss for nonstationary open dynamical systems.
First defined in~\cite{Mohapatra_Ott_2014}, the notion of memory loss we use is a nonstationary analog of decay of correlations and may be described informally as follows.
Let $\varphi_{0}$ and $\psi_{0}$ be two initial probability densities defined on $X$.
Let $\varphi_{t}$ and $\psi_{t}$ denote the evolved densities under the action of the nonstationary open dynamical system $((F_{m}), (H_{j}))$.
We say that $((F_{m}), (H_{j}))$ exhibits {\itshape \bfseries conditional memory loss in the statistical sense} if for all probability densities $\varphi_{0}$ and $\psi_{0}$ chosen from a suitable class, we have
\begin{equation*}
\lim_{t \to \infty } \vnorm{ \frac{\varphi_{t}}{\norm{\varphi_{t}}_{L^{1} (\lambda )}} - \frac{\psi_{t}}{\norm{\psi_{t}}_{L^{1} (\lambda )}} }_{L^{1} (\lambda )} = 0.
\end{equation*}
Our results yield memory loss of this type at exponential rates.

\subsection{Setting}
\label{ssection:set1}

We formulate our results for a metric space $\mscr{M}$ of nonsingular maps on $X$ and a collection $\mscr{H}$ of subsets of $X$ (holes), from which we derive nonstationary open dynamical systems.
We assume holes are Borel measurable, so $\mscr{H}$ is contained in the Borel $\sigma$-algebra.
Our analysis involves the sequence of transfer operators $(\mcal{L}_{F_{m}})$ that describes the evolution of densities under the action of $((F_{m}), (H_{j}))$.  
The operator $\mcal{L}_{F_{m}} : L^{1} (\lambda ) \to L^{1} (\lambda )$ is defined by
\begin{equation*}
\int_{X} \mcal{L}_{F_{m}} (\varphi ) \cdot \psi \, \mrm{d} \lambda = \int_{S_{m}} \varphi \cdot (\psi \circ F_{m}) \, \mrm{d} \lambda ,
\end{equation*}
where $\varphi \in L^{1} (\lambda )$ and $\psi \in L^{\infty } (\lambda )$.

We now describe the technical assumptions behind our general results.

\begin{enumerate}[leftmargin=*, topsep=1ex, itemsep=0.7ex, label=\textbf{(A\arabic*)}, ref=A\arabic*, series=brees]

\item
\label{A:1}
(Inequality of Lasota-Yorke/Doeblin-Fortet type)
There exists a `strong' seminorm $| \cdot |_{s}$ on $L^{1} (\lambda )$ such that the following holds.
For every $\hat{g} \in \mscr{M}$, there exist a neighborhood $\mscr{N} (\hat{g})$ of $\hat{g}$ in $\mscr{M}$, a time $T_{1} \in \mbb{N}$, and constants $\theta_{\LY}  \in (0,1)$ and $C_{\LY} > 0$ such that for every sequence $(\hat{f}_{i})_{i=1}^{\infty }$ in $\mscr{N} (\hat{g})$ and every sequence $(H_j)$ in $\mscr{H}$, the corresponding nonstationary open system $((F_{m}), (H_{j}))$ satisfies
\begin{equation}
\label{e:LY-DF}
| \mcal{L}_{F_{k T_{1}}} (\varphi ) |_{s} \leqs \theta_{\LY}^{k T_{1}} |\varphi |_{s} + C_{\LY} \vnorm{ \varphi }_{L^{1} (\lambda )}
\end{equation}
for every $\varphi \in L^{1} (\lambda )$ and every $k \in \mbb{N}$.

\end{enumerate}

\begin{remark}
  Crucially, the existence of such a Lasota-Yorke inequality depends on the nature of both the maps and the holes.
  Holes are always Borel-measurable, but assuming they possess certain geometric regularity facilitates the derivation of Lasota-Yorke inequalities in concrete settings.
  We do this in Section~\ref{s:app}, wherein we work with a class of piecewise-smooth expanding systems.
\end{remark}

\begin{remark}
When no holes are present and we iterate a single map, in many cases a Lasota-Yorke inequality may be established for the map itself.
However, when iterating a single map in the presence of fixed holes, it can become necessary to derive a Lasota-Yorke inequality for a suitable power of the map.
This is why we formulate~\pref{A:1} for $k T_{1}$-fold compositions of maps in the nonstationary open case.
\end{remark}

\begin{remark}
In our primary theorems, we only require that~\eqref{e:LY-DF} hold for finitely many values of $k$.
\end{remark}

\begin{enumerate}[resume*=brees]

\item
\label{A:2}
(Strength of $|\cdot|_s$) 
There exist $M_{1},M_{2} > 0$ such that the following hold.
First, for every finite partition $\mscr{Q}$ of $X$ into measurable sets and every $\varphi \in L^1(\lambda)$ with $|\varphi|_s < \infty$, we have
\begin{equation}
\label{semi:decomp}
|\varphi|_s \geqs M_{1} \sum_{X' \in \mscr{Q}} |\varphi | X' |_s.
\end{equation}
Second, for every measurable subset $X' \subset X$ of positive measure and every $\varphi \in L^1(\lambda)$ with $|\varphi|_s < \infty$, we have
\begin{equation}
\label{semi:strong}
|\varphi | X' |_s \geqs M_{2} \left| \frac{1}{\lambda(X')} \int_{X'} \vp(x) \, \mrm{d}\lambda(x) - \vp(z) \right|
\end{equation}
for $\lambda$-a.e. $z \in X'$.

\end{enumerate}

\begin{remark}
While~\pref{A:1} is essential, \pref{A:2} is merely a technical assumption used only for the proof of Lemma~\ref{L:control}.
In any concrete setting, one may adjust or drop~\pref{A:2} provided one can establish Lemma~\ref{L:control}.
In particular, the estimate in Lemma~\ref{L:control} involves the constant $M = M_{1} M_{2}$.
For the oscillation seminorm we use in Section~\ref{s:app}, we show that our primary theorems allow $M$ to depend on the partition $\mscr{Q}$.
\end{remark}

Assumptions~\pref{A:1} and \pref{A:2} alone do not imply conditional memory loss.
Even for iterates of a single map with no holes, \pref{A:1} and \pref{A:2} do not imply ergodicity, and memory loss in this situation (decay of the correlation function) is equivalent to mixing.
We therefore formulate a mixing condition in terms of a sequence of (not necessarily dynamical!) partitions of the phase space.

\begin{definition}[A class of mixing maps with respect to $\lambda $]
\label{defn:mix}
Let $\zeta_{1} \in (0,1)$ and $\zeta_{2} \in (1, \infty )$.
We say the measurable map $\hat{g} : X \to X$ belongs to $\mscr{E} (\zeta_{1}, \zeta_{2})$ if there exists a sequence $(\mscr{Q}_{n})_{n=1}^{\infty }$ of finite partitions of $X$ into measurable sets such that
\begin{equation*}
\diam_{\lambda}(\mscr{Q}_n) \defas \sup_{A\in\mscr{Q}_n}\lambda(A) \to 0
\end{equation*}
as $n\to \infty$, and for every $n \in \mbb{N}$ there exists a time $E(\mscr{Q}_{n}, \zeta_{1}, \zeta_{2})$ for which
\begin{equation}
\label{e:Leb-mix}
\zeta_{1} < \frac{ \lambda (J_{1} \cap \hat{g}^{-i} (J_{2})) }{\lambda (J_{1}) \lambda (J_{2}) } < \zeta_{2}
\end{equation}
for every $J_{1}, J_{2} \in \mscr{Q}_{n}$ and every $i \geqs E (\mscr{Q}_{n}, \zeta_{1}, \zeta_{2})$.
We call the partitions $\mscr{Q}_{n}$ \emph{mixing partitions} associated with $\hat{g}$.
\end{definition}

\begin{remark}
We emphasize that the partitions $\mscr{Q}_{n}$ need not be dynamical partitions.
In any concrete setting, one could choose geometrically simple partitions (e.g. partitions of the torus into hypercubes).
In our primary theorems, we only require that~\eqref{e:Leb-mix} hold for finitely many iterates of the map.
Consequently, the mixing condition is a checkable condition with respect to our primary results.
\end{remark}

Because we work in a nonstationary setting, it is advantageous to work with maps for which the mixing condition expressed in Definition~\ref{defn:mix} is stable.
We make this precise with the following assumption.

\begin{enumerate}[resume*=brees]

\item
\label{A:3}
(Stability of mixing for maps in $\mscr{M} \cap \mscr{E}(\zeta_1, \zeta_2)$)
Let $\zeta_{1} \in (0,1)$, $\zeta_{2} \in (1, \infty )$, and $\hat{g}\in \mscr{M} \cap \mscr{E}(\zeta_1,\zeta_2)$.
We assume that for every mixing partition $\mscr{Q}_{n}$ associated with $\hat{g}$ and every $S \geqs E(\mscr{Q}_n,\zeta_1,\zeta_2)$, there exist $\delta >0$ and $\varepsilon > 0$ such that for every sequence $(\hat{f}_k)_{k=1}^S$ in $B(\hat{g},\delta)$ and every sequence $(H_j)_{j=1}^S$ of holes in $\mscr{H}$ with $\lambda(H_j) \leqs \varepsilon$ for all $1 \leqs j \leqs S$, we have
\begin{equation*}
\zeta_1 < \frac{\lambda(J_1 \cap (F_S)^{-1}(J_2))}{\lambda(J_1)\lambda(J_2)} < \zeta_2
\end{equation*}
for all $J_1,J_2 \in \mscr{Q}_n.$

\end{enumerate}

\begin{definition}
  We call $\hat{g} \in \mscr{M} \cap \mscr{E} (\zeta_1, \zeta_2)$ \emph{stably mixing} if $\hat{g}$ satisfies~\pref{A:3}.
\end{definition}

\subsection{Primary general theorems}
\label{ss:primary-theorems}

Conditional memory loss in the statistical sense involves the asymptotic decay of $\norm{ \mcal{R}_{F_{m}} (\varphi ) - \mcal{R}_{F_{m}} (\psi ) }_{L^1(\lambda)}$, where $\mcal{R}_{F_{m}}$ is the normalized transfer operator defined by
\begin{equation*}
\mcal{R}_{F_{m}} (\varphi ) = \frac{\mcal{L}_{F_{m}} (\varphi )}{\vnorm{ \mcal{L}_{F_{m}} (\varphi ) }_{L^{1} (\lambda )}},
\end{equation*}
provided the denominator is nonzero.
Notice that while $\mcal{L}_{F_{m}}$ is linear, the operator $\mcal{R}_{F_{m}}$ is not linear.
Our main results concern the action of $\mcal{R}_{F_{m}}$ on a suitable subset of the space
\begin{equation*}
\mcal{D} = \vset{ \varphi \in L^{1} (\lambda ) : |\varphi |_{s} < \infty , \; \: \varphi \geqs 0, \; \: \norm{\varphi }_{L^{1} (\lambda )} = 1 }.
\end{equation*}
In particular, we introduce a nontrivial convex cone $\mcal{C}_{a} \subset L^{1} (\lambda )$ and then examine the action of $\mcal{R}_{F_{m}}$ on $\mcal{D} \cap \mcal{C}_{a}$.
See Remark~\ref{remark:WhyCone} for intuition about why the cone is needed, and Eq.~\eqref{e:DefCone} for the precise specification of the cone.

We formulate two theorems - a local result and a global result.
The local result asserts that conditional memory loss occurs at an exponential rate for the nonstationary open system provided the holes are small (they may move without restriction), and the maps involved are all close to a single stably mixing base map.
The global result allows us to move far and wide within the space of maps when selecting those that form the nonstationary system, provided we move slowly.

\begin{MainTheorem}[General local result]
\label{thrm:local}
Assume $\mscr{M}$ and $\mscr{H}$ satisfy (\ref{A:1})-(\ref{A:2}). 
Let $\zeta_1 \in (0,1)$, $\zeta_2 \in (1,\infty)$, and suppose that $\hat{g}\in \mscr{M} \cap \mscr{E}(\zeta_1,\zeta_2)$ satisfies~\pref{A:3}.
There exists a convex cone $\mcal{C}_{a} \subset L^{1} (\lambda )$, as well as associated constants $\delta_0 > 0$, $\vep_{0} > 0$, $\Lambda \in (0,1)$, and $C_{0} > 0$, such that the following holds.
Let $(\hat{f}_{i})_{i = 1}^{\infty }$ be any sequence of maps in the ball $B(\hat{g},\delta_0),$ and let $(H_{j})_{j=1}^{\infty }$ be any sequence of holes in $\mscr{H}$ such that $\lambda (H_{j}) \leqs \vep_{0}$ for every $j \in \mbb{N}$.
The resultant nonstationary open dynamical system $((F_{m}), (H_{j}))$ exhibits conditional memory loss at an exponential rate: For every $\varphi , \psi \in \mcal{D} \cap \mcal{C}_{a}$, we have
\begin{equation*}
\vnorm{ \mcal{R}_{F_{m}} (\varphi ) - \mcal{R}_{F_{m}} (\psi ) }_{L^{1} (\lambda )} \leqs C_{0} \Lambda^{m}
\end{equation*}
for all $m \in \mbb{N}$.
\end{MainTheorem}

\begin{remark}
\label{remark:WhyCone}
Functions $\varphi \in \mcal{D} \cap \mcal{C}_{a}$ satisfy $| \varphi |_{s} \leqs a$ (though this is not the defining condition for the cone).
One might expect the conditional memory loss in Theorem~\ref{thrm:local} to hold for all functions $\varphi \in \mcal{D}$ (that is, the condition $| \varphi |_{s} < \infty$ is sufficient).
However, control on the strong seminorm is needed because of the presence of holes.
For example, suppose we work with piecewise-smooth expanding circle maps and let $| \cdot |_{s}$ be the total variation.
In this case, a probability density function with large total variation could have very small support.
All of the mass associated with such a probability density function could fall into even a small hole in a single step.
\end{remark}

\begin{MainTheorem}[General global result]
\label{thrm:global}

Assume $\mscr{M}$ and $\mscr{H}$ satisfy (\ref{A:1})-(\ref{A:2}).  
Let $\zeta_1 \in (0,1)$ and $\zeta_2 \in (1,\infty).$  
Let $\hat{\gamma}: [a,b] \to \mscr{M} \cap \mscr{E}(\zeta_1,\zeta_2)$ be a continuous map such that $\hat{\gamma} (t)$ satisfies~\pref{A:3} for all $t \in [a,b]$.
There exists a convex cone $\mcal{C}_{a_{\hat{\gamma}}} \subset L^1(\lambda)$, as well as associated constants $\Sigma_{\hat{\gamma}} > 0$, $\vep_{\hat{\gamma}} > 0$, $\Lambda_{\hat{\gamma}} \in (0,1)$, and $C_{\hat{\gamma}} > 0$, such that the following holds.  
For any map $\gamma_{\mscr{H}} : [a,b] \to \mscr{H}$ and any finite or infinite increasing sequence $(t_j)$ in $[a,b]$, if $\max_j(t_{j+1} - t_j) \leqs \Sigma_{\hat{\gamma}}$ and if $\lambda(\gamma_{\mscr{H}}(t_j)) \leqs \vep_{\hat{\gamma}}$ for all relevant $j$, then the nonstationary open dynamical system defined by the maps
\begin{equation*}
\widehat{F}_m = \hat{\gamma}(t_m) \circ \cdots \circ \hat{\gamma}(t_1)
\end{equation*}
and holes $H_j = \gamma_{\mscr{H}}(t_j)$ exhibits conditional memory loss at an exponential rate.  
Precisely, for every $\varphi, \psi \in \mcal{D} \cap \mcal{C}_{a_{\hat{\gamma}}}$, we have
\begin{equation*}
\vnorm{\mcal{R}_{F_m}(\phi) - \mcal{R}_{F_m}(\psi)}_{L^1(\lambda)} \leqs C_{\hat{\gamma}} \Lambda_{\hat{\gamma}}^m
\end{equation*}
for all relevant $m$.
\end{MainTheorem}

\begin{remark}
Theorem~\ref{thrm:global} is but a sample of possible global results - many generalizations may be formulated.
For example, $\hat{\gamma}$ may admit discontinuities.
\end{remark}

\begin{remark}
The quantity $\Sigma_{\hat{\gamma}}$ is a measure of the speed at which the curve $\hat{\gamma}$ is traversed.
Theorem~\ref{thrm:global} yields conditional memory loss at an exponential rate provided the traversal speed is sufficiently small.
Dobbs and Stenlund~\cite{Dobbs-Stenlund-2017} study the quasi-static $\Sigma_{\hat{\gamma}} \to 0$ limit in the context of nonstationary systems on the circle.
\end{remark}

\subsection{Proof of Theorem \ref{thrm:local}}\label{ssection:pf} The proof involves analyzing the action of $(\mcal{L}_{F_{m}})$ on a suitable convex cone $\mathcal{C}_{a}$ of functions in $L^{1}(\lambda)$ that is defined in terms of the seminorm $|\cdot|_{s}.$  We begin with parameter selection. Let $M = M_{1} M_{2}$, where $M_{1}$ and $M_{2}$ are the constants in~\pref{A:2}. First, introduce 
\begin{itemize}[leftmargin=*, labelindent=\parindent]
\item
$a$: aperture of the cone $\mathcal{C}_{a}$,
\item
$T$: a mixing time,
\item
$n_{0}$: controls the measure-theoretic diameter of the mixing partition $\mscr{Q}_{n_{0}}$ associated with $\hat{g}$,
\item
$\sigma$: facilitates showing that suitable compositions of transfer operators map $\mathcal{C}_{a}$ strictly inside itself,
\end{itemize}
so that~(\ref{P:1})--(\ref{P:3}) are satisfied simultaneously.  
Second, choose $\delta_0$ and $\vep_0.$
We now carry this out in detail.

The important parameter inequalities are the following.
\begin{enumerate}[leftmargin=*, topsep=1ex, itemsep=0.7ex, label=\textbf{(P\arabic*)}, ref=P\arabic*]
\item\label{P:1}
$0<\sigma<1$.

\item\label{P:2}
$T \in \mbb{N}$ is a positive integer multiple of $T_{1}$ and $T\geqs E(\mscr{Q}_{n_0},\zeta_{1},\zeta_{2})$.

\item\label{P:3}
$a>0$ satisfies
\begin{equation*}
\zeta_{1}-\frac{\zeta_{2}a}{M}\cdot\diam_{\lambda}(\mscr{Q}_{n_0}) > 0, \qquad
(a \theta_{LY}^{T} + C_{\LY}) \left( \zeta_{1} - \frac{\zeta_{2}a}{M} \cdot \diam_{\lambda}(\mscr{Q}_{n_0}) \right)^{-1} \leqs \sigma a.
\end{equation*}
\end{enumerate}
To see that (\ref{P:1})--(\ref{P:3}) may be satisfied simultaneously, proceed in the following order.
\begin{itemize}[leftmargin=*, labelindent=\parindent]
\item
Select $\sigma \in (0,1)$.

\item
Choose $T$ sufficiently large so that $\theta_{LY}^{T}/(\zeta_{1}/2)<\sigma$.

\item
Choose $a$ sufficiently large so that 
\begin{equation*}
\frac{a\theta_{LY}^{T}+C_{\LY}}{\zeta_{1}/2}\leqs \sigma a.
\end{equation*}

\item Choose $\diam_{\lambda}(\mscr{Q}_{n_0})$ sufficiently small so that $\zeta_{2} a M^{-1} \cdot\diam_{\lambda}(\mscr{Q}_{n_0})\leqs \zeta_{1}/2$.

\item Increase $T$ (if necessary) so that $T\geqs E(\mscr{Q}_{n_0},\zeta_{1},\zeta_{2}).$
\end{itemize}
We now choose $\delta_{0}$ and $\varepsilon_{0}$ using (\ref{A:3}) with base map $\hat{g}$ and time $T$.
Further, assume $\delta_{0}$ is sufficiently small so that $B(\hat{g},\delta_0) \subset \mscr{N}(\hat{g})$.
For every sequence $(\hat{f}_k)_{k=1}^T$ in $B(\hat{g},\delta_0)$ and every sequence $(H_{j})_{j=1}^{T}$ of holes in $\mscr{H}$ with $\lambda (H_{j}) \leqs \vep_{0}$ for all $1 \leqs j \leqs T$, we therefore have
\begin{equation}\label{mix:e0}
\zeta_{1} < \frac{\lambda(J_1 \cap (F_T)^{-1}(J_2))}{\lambda(J_1)\lambda(J_2)} < \zeta_2
\end{equation}
for all $J_1,J_2 \in \mscr{Q}_{n_0}.$

\subsection{Invariance of a suitable convex cone}\label{ssection:cone}

Define
\begin{equation}
\label{e:DefCone}
\mathcal{C}_{a}=\vset{\varphi\in L^{1}(\lambda) : \varphi\geqs 0 , \;\: \varphi\not\equiv 0 , \;\: |\varphi|_{s}\leqs a\mbb{E}[\varphi|\mscr{Q}_{n_0}]}.
\end{equation}
We study the action of $\mathcal{L}_{F_{m}}$ on $\mathcal{C}_{a}.$  For positive integers $m>i$, define
\begin{equation*}
\hat{F}_{m,i} = \hat{f}_m \circ \hat{f}_{m-1} \circ \cdots \circ \hat{f}_{i} , \;\: F_{m,i} = f_m \circ f_{m-1} \circ \cdots \circ f_{i},
\end{equation*}
where $f_{k}$ is the open system corresponding to $\hat{f}_{k}$ for $i\leqs k\leqs m.$
Note that $\hat{F}_{m,1} = \hat{F}_{m}$ and $F_{m,1} = F_{m}$.

\begin{lemma}\label{L:control}
Assume the setting of Section ~\ref{ssection:pf}.  For every $\varphi\in \mathcal{C}_{a}$ and $i\in \mbb{N}$ we have
\begin{equation}
\label{E:control}
\left( \zeta_{1} - \frac{\zeta_{2}a}{M} \cdot \diam_{\lambda}(\mscr{Q}_{n_0}) \right) \int_{X} \varphi \, \mrm{d} \lambda \leqs \mbb{E} [\mathcal{L}_{F_{i+T-1,i}} (\varphi) | \mscr{Q}_{n_0}] \leqs \zeta_2 \left( 1 + \frac{a}{M} \cdot \diam_{\lambda} (\mscr{Q}_{n_0}) \right) \int_{X} \varphi \, \mrm{d} \lambda.
\end{equation}
\end{lemma}

\begin{remark}
In light of~\pref{P:3}, Lemma~\ref{L:control} asserts that for any density $\varphi \in \mcal{D} \cap \mcal{C}_{a}$, the open transfer operator $\mcal{L}_{F_{i+T-1,i}}$ maps $\varphi$ to a function with uniformly positive conditional expectation. 
\end{remark}

\begin{proof}
Write $F = F_{i+T-1,i}.$  For $x\in X$, let $Q(x)$ denote the element of $\mscr{Q}_{n_0}$ that contains $x$.  We have
\begin{align}\label{int:b1}
\mbb{E}[\mathcal{L}_F(\varphi)|\mscr{Q}_{n_0}](x) &= \frac{1}{\lambda(Q(x))}\int_{Q(x)} \mathcal{L}_F(\varphi) \, \mrm{d} \lambda \nonumber\\
&= \frac{1}{\lambda(Q(x))}\int_{F^{-1}(Q(x))}\varphi \, \mrm{d}\lambda \nonumber\\
&= \frac{1}{\lambda(Q(x))}\sum_{Q'\in \mscr{Q}_{n_0}}\int_{Q'\cap F^{-1}(Q(x))} \varphi(z) \, \mrm{d}\lambda(z). 
\end{align}
Fix $Q' \in \mscr{Q}_{n_{0}}$.
Using~\pref{A:2}, for almost every $z\in Q'\cap F^{-1}(Q(x))$ we have
\begin{align}\label{int:b2}
\varphi(z) &\geqs \frac{1}{\lambda(Q')}\int_{Q'} \varphi \, \mrm{d}\lambda - \frac{1}{M_2} |\varphi | Q'|_s \nonumber\\
&= \frac{1}{\lambda(Q')}\left(\int_{Q'} \varphi \, \mrm{d}\lambda -\frac{\lambda(Q')}{M_2} |\varphi | Q'|_s \right).
\end{align}
Summing over $Q' \in \mscr{Q}_{n_{0}}$ and using~(\ref{int:b1}), (\ref{int:b2}), (\ref{mix:e0}), \pref{A:2}, and the fact that $\varphi \in \mcal{C}_{a}$ implies $| \varphi |_{s} \leqs a \int_{X} \varphi \, \mrm{d} \lambda$, we have
\begin{align*}
\mbb{E}[\mathcal{L}_F(\varphi)|\mscr{Q}_{n_0}](x) &\geqs \frac{1}{\lambda(Q(x))} \sum_{Q'\in \mscr{Q}_{n_0}}\int_{Q'\cap F^{-1}(Q(x))}\frac{1}{\lambda(Q')}\left(\int_{Q'} \varphi \, \mrm{d}\lambda - \frac{\lambda(Q')}{M_2} |\varphi | Q'|_s \right) \mrm{d} \lambda(z)\\
&= \sum_{Q'\in \mscr{Q}_{n_0}} \frac{\lambda(Q'\cap F^{-1}(Q(x)))}{\lambda(Q(x))\lambda(Q')}\left(\int_{Q'} \varphi \, \mrm{d} \lambda - \frac{\lambda(Q')}{M_2} |\varphi | Q'|_s \right)\\
&\geqs \zeta_1\int_X \varphi \, \mrm{d}\lambda - \frac{\zeta_2}{M_{1} M_{2}}\cdot \diam_{\lambda}(\mscr{Q}_{n_0})\cdot|\varphi|_s\\
&\geqs \zeta_1\int_X \varphi \, \mrm{d}\lambda - \frac{\zeta_2a}{M}\cdot \diam_{\lambda}(\mscr{Q}_{n_0})\cdot\int_X \varphi \, \mrm{d}\lambda\\
&= \left(\zeta_1 - \frac{\zeta_2a}{M}\cdot\diam_{\lambda}(\mscr{Q}_{n_0})\right)\int_X \varphi \, \mrm{d}\lambda.
\end{align*}
The upper bound
\begin{equation*}
\mbb{E}[\mathcal{L}_F(\varphi)|\mscr{Q}_{n_0}](x) \leqs \zeta_2\left(1+\frac{a}{M}\cdot\diam_{\lambda}(\mscr{Q}_{n_0})\right)\int_X \varphi \, \mrm{d}\lambda
\end{equation*}
follows from an analogous line of reasoning.
\end{proof}

\begin{proposition}\label{cone:contract}
In the setting of Lemma ~\ref{L:control}, for every $i\in\mbb{N}$, we have
\begin{equation*}
\mathcal{L}_{F_{i+T-1,i}}(\mcal{C}_a)\subset \mcal{C}_{\sigma a}.
\end{equation*}
\end{proposition}

\begin{proof}
Write $F=F_{i+T-1,i}$ and let $\varphi \in \mathcal{C}_{a}.$  Using~(\ref{e:LY-DF}), (\ref{E:control}), a basic property of the conditional expectation, and control~\pref{P:3} on the cone aperture $a$, we have
\begin{align*}
|\mathcal{L}_F(\varphi)|_{s} &\leqs \theta_{LY}^{T}|\varphi|_{s} + C_{\LY}\vnorm{\varphi}_{L^{1}(\lambda)}\\
&\leqs (a\theta_{LY}^{T}+C_{\LY})\vnorm{\varphi}_{L^{1}(\lambda)}\\
&\leqs (a \theta_{LY}^{T} + C_{\LY}) \left( \zeta_{1} - \frac{\zeta_{2} a}{M} \cdot \diam_{\lambda}(\mscr{Q}_{n_0}) \right)^{-1} \mbb{E}[\mathcal{L}_{F} (\varphi)|\mscr{Q}_{n_0}]\\
&\leqs \sigma a\mbb{E}[\mathcal{L}_{F}(\varphi)|\mscr{Q}_{n_0}].
\end{align*}
\end{proof}

\subsection{Cones, Hilbert metrics, and positive operators}

Following~\cite{Liverani_Deschamps_2003} and~\cite{LSV_1998}, we review a theory of cones developed by Birkhoff~\cite{Birkhoff-Lattice-1979}.
We will use this theory to show that $\mathcal{L}_{F_{i+T-1,i}}$ acts contractively on $\mcal{C}_{a}$ with respect to a projective metric known as the Hilbert metric.

\begin{definition} Let $\mscr{V}$ be a real vector space.  A \textit{convex cone} is a subset $\mathcal{C}\subset \mscr{V}$ with the following properties.
\begin{enumerate}[topsep=1ex, itemsep=0.5ex, label=(\alph*), ref=\alph*, series=brees]

\item
$\mathcal{C}\cap\mathcal{-C}=\emptyset$.

\item
$\gamma\mathcal{C} = \mathcal{C}$ for all $\gamma > 0$.

\item
$\mathcal{C}$ is a convex set.

\item
For all $\varphi,\psi \in \mathcal{C}$, every $c\in \mbb{R}$, and every sequence $(c_n)$ in $\mbb{R}$ such that $c_n \to c$, if $\varphi - c_n\psi \in \mathcal{C}$ for all $n$, then $\varphi - c\psi \in \mathcal{C}\cup\{0\}.$

\end{enumerate}
\end{definition}

\begin{definition} Let $\mathcal{C}$ be a convex cone.  The {\itshape Hilbert metric} $d_{\mathcal{C}}$ is defined on $\mathcal{C}$ by 
\begin{equation*}
d_{\mathcal{C}} (\varphi,\psi) = \log\left(\frac{\inf\{c>0 : c\varphi - \psi \in \mathcal{C}\}}{\sup\{r>0 : \psi - r\varphi \in \mathcal{C}\}}\right).
\end{equation*}
\end{definition}

The following result of Birkhoff asserts that in the current context, a positive linear operator is a contraction in the Hilbert metric provided the diameter of the image is finite.

\begin{theorem}[\cite{Birkhoff-Lattice-1979}]
\label{Birk:cone}
Let $\mscr{V}_1$ and $\mscr{V}_2$ be real vector spaces containing convex cones $\mathcal{C}_1$ and $\mathcal{C}_2$, respectively.  Let $\mathcal{L}:\mscr{V}_1 \to \mscr{V}_2$ be a positive linear operator, meaning $\mathcal{L}(\mathcal{C}_1) \subset \mathcal{C}_2.$  Define 
\begin{equation*}
\Delta = \sup_{\varphi^*,\psi^* \in \mathcal{L}(\mathcal{C}_1)} d_{\mathcal{C}_{2}} (\varphi^*,\psi^*).
\end{equation*}
Then for all $\varphi,\psi \in \mathcal{C}_1,$ we have
\begin{equation*}
d_{\mathcal{C}_2}(\mathcal{L}\varphi,\mathcal{L}\psi) \leqs \tanh\left(\frac{\Delta}{4}\right) d_{\mathcal{C}_1(\varphi,\psi)}.
\end{equation*}
\end{theorem}

We conclude this review by relating the Hilbert metric to \textit{adapted} norms on $\mscr{V}.$

\begin{proposition}[\cite{LSV_1998}]
\label{norm:adapt}
Let $\mathcal{C} \subset \mscr{V}$ be a convex cone and let $\vnorm{\cdot}$ be an adapted norm on $\mscr{V}$; that is, a norm such that for all $\varphi,\psi \in \mscr{V}$, if $\psi - \varphi \in \mathcal{C}$ and $\psi + \varphi \in \mathcal{C}$, then $\vnorm{\varphi} \leqs \vnorm{\psi}.$
Then for all $\varphi,\psi \in \mathcal{C},$ we have
\begin{equation*}
\vnorm{\varphi}=\vnorm{\psi} \Longrightarrow \vnorm{\varphi - \psi} \leqs (e^{d_{\mathcal{C}}(\varphi,\psi)} - 1)\vnorm{\varphi}.
\end{equation*}
\end{proposition}

\subsection{Completion of the proof of Theorem~\ref{thrm:local}}

\begin{proposition}\label{cone:dist}
Assume the setting of Proposition~\ref{cone:contract}.  For every $i\in\mbb{N}$ and for all $\varphi,\psi \in \mathcal{C}_a$, we have
\begin{equation}\label{cone:bound}
\begin{aligned}
d_{\mcal{C}_a}(\mathcal{L}_{F_{i+T-1,i}}(\varphi),\mathcal{L}_{F_{i+T-1,i}}(\psi)) \leqs \Delta_0 &\defas 2 \log \left( \frac{1+\sigma}{1-\sigma} \right)
\\
&\quad {}+ 2 \log \left( \frac{\zeta_2 (1 + a M^{-1} \cdot \diam_{\lambda}(\mscr{Q}_{n_0}))}{\zeta_{1} - \zeta_{2} a M^{-1} \cdot \diam_{\lambda}(\mscr{Q}_{n_0})} \right).
\end{aligned}
\end{equation}
\end{proposition}

\begin{proof} Let $\varphi^*,\psi^* \in \mathcal{C}_{\sigma a}$.  Suppose $c > 0.$  We have
\begin{align*}
|c\varphi^* - \psi^*|_s &\leqs c|\varphi^*|_s + |\psi^*|_s \\
                        &\leqs c\sigma a\mbb{E}[\varphi^*|\mscr{Q}_{n_0}] + \sigma a\mbb{E}[\psi^*|\mscr{Q}_{n_0}].
\end{align*}
Therefore, $c\varphi^* - \psi^* \in \mathcal{C}_a$ if 
\begin{equation*}
c\sigma a\mbb{E}[\varphi^*|\mscr{Q}_{n_0}] + \sigma a\mbb{E}[\psi^*|\mscr{Q}_{n_0}] \leqs a\mbb{E}[c\varphi^* - \psi^*|\mscr{Q}_{n_0}].
\end{equation*}
This is equivalent to 
\begin{equation}\label{eqn:upbound}
\left(\frac{1+\sigma}{1-\sigma}\right) \left(\frac{\mbb{E}[\psi^*|\mscr{Q}_{n_0}]}{\mbb{E}[\varphi^*|\mscr{Q}_{n_0}]}\right) \leqs c.
\end{equation}
Arguing analogously, for $r > 0$ we have $\psi^* - r\varphi^* \in \mathcal{C}_a$ if 
\begin{equation}\label{eqn:lowbound}
r \leqs \left(\frac{1-\sigma}{1+\sigma}\right) \left(\frac{\mbb{E}[\psi^*|\mscr{Q}_{n_0}]}{\mbb{E}[\varphi^*|\mscr{Q}_{n_0}]}\right).
\end{equation}
Using $\mathcal{L}_{F_{i+T-1,i}} (\mathcal{C}_{a}) \subset \mathcal{C}_{\sigma a}$ by Proposition~\ref{cone:contract} and setting $\varphi^{*} = \mathcal{L}_{F_{i+T-1,i}} ( \varphi )$ and $\psi^{*} = \mathcal{L}_{F_{i+T-1,i}} ( \psi )$, bounds (\ref{eqn:upbound}) and (\ref{eqn:lowbound}) imply
\begin{align}\label{Hilbert:control}
d_{\mathcal{C}_a}(\mathcal{L}_{F_{i+T-1,i}}(\varphi),\mathcal{L}_{F_{i+T-1,i}}(\psi)) &\leqs \log\left(\left(\frac{1+\sigma}{1-\sigma}\right)\sup_{x\in X}\frac{\mbb{E}[\psi^*|\mscr{Q}_{n_0}]}{\mbb{E}[\varphi^*|\mscr{Q}_{n_0}]}\right) - \log\left(\left(\frac{1-\sigma}{1+\sigma}\right)\inf_{x\in X}\frac{\mbb{E}[\psi^*|\mscr{Q}_{n_0}]}{\mbb{E}[\varphi^*|\mscr{Q}_{n_0}]}\right) \nonumber\\
&= 2\log\left(\frac{1+\sigma}{1-\sigma}\right) + \log\left(\sup_{x\in X}\frac{\mbb{E}[\psi^*|\mscr{Q}_{n_0}]}{\mbb{E}[\varphi^*|\mscr{Q}_{n_0}]}\right) -  \log\left(\inf_{x\in X}\frac{\mbb{E}[\psi^*|\mscr{Q}_{n_0}]}{\mbb{E}[\varphi^*|\mscr{Q}_{n_0}]}\right).
\end{align}
Estimates (\ref{E:control}) and (\ref{Hilbert:control}) imply (\ref{cone:bound}) with
\begin{equation*}
\Delta_0 = 2\log\left(\frac{1+\sigma}{1-\sigma}\right) + 2\log\left(\frac{\zeta_2 (1 + a M^{-1} \cdot \diam_{\lambda}(\mscr{Q}_{n_0}))}{\zeta_{1} - \zeta_{2} a M^{-1} \cdot \diam_{\lambda}(\mscr{Q}_{n_0})}\right).
\end{equation*}
\end{proof}

Theorem~\ref{Birk:cone} and Proposition~\ref{cone:dist} now imply that $\mcal{L}_{F_{i+T-1,i}}$ contracts $\mcal{C}_{a}$.

\begin{corollary}\label{cor:conedist} Assume the setting of Proposition~\ref{cone:dist}.  For every $i\in \mbb{N}$ and for all $\varphi,\psi \in \mathcal{C}_a,$ we have
\begin{equation}\label{L:contract}
d_{\mathcal{C}_a}(\mathcal{L}_{F_{i+T-1,i}}(\varphi),\mathcal{L}_{F_{i+T-1,i}}(\psi)) \leqs \tanh\left(\frac{\Delta_0}{4}\right) d_{\mathcal{C}_a} (\varphi,\psi).
\end{equation}
\end{corollary}

We need one final ingredient in order to prove Theorem~\ref{thrm:local} - a Lipschitz estimate involving the normalized transfer operator $\mathcal{R}$.

\begin{lemma}\label{Lip:est}
Assume the setting of Corollary~\ref{cor:conedist}.  There exists $C_{\Lip}>0$ such that for all integers $n$ satisfying $1 \leqs n < T,$ for every $i \in \mbb{N},$ and for all $\varphi,\psi \in \mathcal{D}\cap\mathcal{C}_a,$ we have
\begin{equation}\label{Lip:ineq}
\vnorm{\mathcal{R}_{F_{i+n-1,i}}(\varphi) - \mathcal{R}_{F_{i+n-1,i}}(\psi)}_{L^1(\lambda)} \leqs C_{\Lip}\vnorm{\varphi - \psi}_{L^1(\lambda)}.
\end{equation}
\end{lemma}

\begin{proof}
Write $F= F_{i+n-1,i}$ and $\vnorm{\cdot} = \vnorm{\cdot}_{L^1(\lambda)}.$  Let $\varphi,\psi\in\mathcal{D}\cap\mathcal{C}_a.$  We have
\begin{align*}
\vnorm{\mathcal{R}_F(\varphi) - \mathcal{R}_F(\psi)} &= \vnorm{\frac{\mathcal{L}_F(\varphi)}{\vnorm{\mathcal{L}_F(\varphi)}} - \frac{\mathcal{L}_F(\psi)}{\vnorm{\mathcal{L}_F(\psi)}}}\\
&= \vnorm{\frac{\mathcal{L}_F(\varphi)}{\vnorm{\mathcal{L}_F(\varphi)}} - \frac{\mathcal{L}_F(\varphi)}{\vnorm{\mathcal{L}_F(\psi)}} + \frac{\mathcal{L}_F(\varphi)}{\vnorm{\mathcal{L}_F(\psi)}} - \frac{\mathcal{L}_F(\psi)}{\vnorm{\mathcal{L}_F(\psi)}}}\\
&\leqs \left|\frac{\vnorm{\mathcal{L}_F(\psi)} - \vnorm{\mathcal{L}_F(\varphi)}}{\vnorm{\mathcal{L}_F(\varphi)}\cdot\vnorm{\mathcal{L}_F(\psi)}}\right|\vnorm{\mathcal{L}_F(\varphi)} + \frac{1}{\vnorm{\mathcal{L}_F(\psi)}}\vnorm{\mathcal{L}_F(\varphi) - \mathcal{L}_F(\psi)}\\
&\leqs 2\left(\zeta_1 - \frac{\zeta_2a}{M} \cdot \diam_{\lambda} (\mscr{Q}_{n_0})\right)^{-1}\vnorm{\mathcal{L}_F(\varphi) - \mathcal{L}_F(\psi)}\\
&\leqs 2\left(\zeta_1 - \frac{\zeta_2a}{M} \cdot \diam_{\lambda} (\mscr{Q}_{n_0})\right)^{-1}\vnorm{\varphi - \psi},
\end{align*}
using (\ref{E:control}), the fact that $\norm{\mathcal{L}_{F_{i+j-1,i}} (\psi)}$ decreases monotonically with $j$, and the fact that $\vnorm{\mathcal{L}_F(\gamma)} \leqs \vnorm{\gamma}$ for every $\gamma \in L^1(\lambda).$  
Set 
\begin{equation*}
C_{\Lip} = 2\left(\zeta_1 - \frac{\zeta_2a}{M} \cdot \diam_{\lambda} (\mscr{Q}_{n_0})\right)^{-1}.
\end{equation*}
\end{proof}

We now finish the proof of Theorem~\ref{thrm:local}.  Write $\vnorm{\cdot}$ for the $L^1$-norm.  Let $\varphi,\psi\in\mathcal{D}\cap\mathcal{C}_a$.  Let $m \in \mbb{N}$ and write $m = kT + n$ where $k\in \mbb{Z}^{\geqs 0}$ and $0 \leqs n < T.$  If $k\geqs 1,$ we have
\begin{align*}
\vnorm{\mathcal{R}_{F_m}(\varphi) - \mathcal{R}_{F_m}(\psi)} &\leqs C_{\Lip}\vnorm{\mathcal{R}_{F_{kT}}(\varphi) - \mathcal{R}_{F_{kT}}(\psi)} &(\ref{Lip:ineq})\\
&\leqs C_{\Lip}(\exp(d_{\mathcal{C}_a}(\mathcal{R}_{F_{kT}}(\varphi),\mathcal{R}_{F_{kT}}(\psi))) - 1) &(\mbox{P}\ref{norm:adapt})\\
&= C_{\Lip}(\exp(d_{\mathcal{C}_a}(\mathcal{L}_{F_{kT}}(\varphi),\mathcal{L}_{F_{kT}}(\psi))) - 1)\\
&\leqs C_{\Lip}\left(\exp\left(\left(\tanh\left(\frac{\Delta_0}{4}\right)\right)^{k-1}d_{\mathcal{C}_a}(\mathcal{L}_{F_T}(\varphi),\mathcal{L}_{F_T}(\psi))\right)-1\right) &(\ref{L:contract})\\
&\leqs C_{\Lip}\Delta_0 e^{\Delta_0}\left(\tanh\left(\frac{\Delta_0}{4}\right)\right)^{k-1} &(\ref{cone:bound})\\
&\leqs C_{\Lip}\Delta_0 e^{\Delta_0}\tanh^{-2}\left(\frac{\Delta_0}{4}\right)\left(\left(\tanh\left(\frac{\Delta_0}{4}\right)\right)^{1/T}\right)^m.
\end{align*}
Consequently, for any $m \in \mbb{N}$ we have
\begin{equation*}
\vnorm{\mathcal{R}_{F_{m}}(\varphi) - \mathcal{R}_{F_m}(\psi)} \leqs C_{\Lip}\max\{\Delta_0,1\} e^{\Delta_0}\tanh^{-2}\left(\frac{\Delta_0}{4}\right)\left(\left(\tanh\left(\frac{\Delta_0}{4}\right)\right)^{1/T}\right)^m.
\end{equation*}
This proves Theorem~\ref{thrm:local} with 
\begin{align}
\label{C:0}C_0 &= C_{\Lip}\max\{\Delta_0,1\} e^{\Delta_0}\tanh^{-2}\left(\frac{\Delta_0}{4}\right),\\
\label{Delta:exp}\Lambda &= \left(\tanh\left(\frac{\Delta_0}{4}\right)\right)^{1/T}.
\end{align}

\subsection{Proof of Theorem~\ref{thrm:global}}  Now that a local result has been established, we are in position to prove a global result.
Assume the setting and hypotheses of Theorem~\ref{thrm:global}.

For every $s \in [a,b],$ Theorem~\ref{thrm:local} applies to the map $\hat{\gamma}(s) \in \mscr{M} \cap \mscr{E}(\zeta_1,\zeta_2).$  We associate the following with this map:
\begin{itemize}[leftmargin=*, labelindent=\parindent]
\item $\delta_0(s)$: size of the neighborhood of $\hat{\gamma}(s)$ in which Theorem~\ref{thrm:local} applies,
\item $\vep_0(s)$: allowable measure of holes,
\item $\mcal{C}_{a(s)}$: cone containing the densities to which Theorem~\ref{thrm:local} applies,
\item $\Lambda(s),C_0(s)$: parameters in (\ref{C:0}) and (\ref{Delta:exp}),
\item $T(s)$: time required for the transfer operators to contract the cone $\mcal{C}_{a(s)}$ (see Corollary~\ref{cor:conedist}),
\item $\mscr{Q}_{n_0(s)}(s)$: mixing partition for $\hat{\gamma}(s)$ used in the proof of Theorem~\ref{thrm:local}.
\end{itemize}
We now leverage the compactness of $[a,b].$  For each $s\in [a,b],$ there exists an open interval $(s-\xi(s),s+\xi(s))$ such that 
\begin{equation*}
\hat{\gamma} \big( (s-\xi(s),s+\xi(s)) \cap [a,b] \big) \subset B(\hat{\gamma}(s),\delta_0(s)).
\end{equation*}
The open cover
\begin{equation*}
\vset{\Big( s - \frac{1}{2} \xi(s), s + \frac{1}{2} \xi(s) \Big) : s \in [a,b]}
\end{equation*}
of $[a,b]$ has a finite subcover
\begin{equation*}
\vset{\Big( s_{i} - \frac{1}{2} \xi(s_{i}), s_{i} + \frac{1}{2} \xi(s_{i}) \Big) : 1 \leqs i \leqs N}.
\end{equation*}

Theorem~\ref{thrm:global} is formulated for discretizations of $\hat{\gamma}.$  The parameter $\Sigma_{\hat{\gamma}}$ controls the `velocity' at which the curve $\hat{\gamma}$ is traversed.  We choose $\Sigma_{\hat{\gamma}}$ to achieve the following:
Each time point $t_{j}$ in the sequence of times that discretizes the curve $\hat{\gamma}$ lies in some element $U_{i(j)}$ of the finite subcover.
We want to ensure that the maps beginning with $\hat{\gamma} (t_{j})$ and ending with $\hat{\gamma} (t_{j + T(s_{i(j)}) - 1})$ all lie in
\begin{equation*}
\hat{\gamma} \big( (s_{i(j)} - \xi (s_{i(j)}), s_{i(j)} + \xi (s_{i(j)})) \cap [a,b] \big) .
\end{equation*}
In light of this, define
\begin{equation*}
\Sigma_{\hat{\gamma}} = \min_{1 \leqs i \leqs N} \frac{\xi (s_{i})}{2 T(s_{i})}.
\end{equation*}

Theorem~\ref{thrm:global} now follows from Theorem~\ref{thrm:local} with
\begin{align*}
\vep_{\hat{\gamma}} &= \min_{1\leqs i\leqs N} \vep_0(s_i),\\
\Lambda_{\hat{\gamma}} &= \max_{1\leqs i\leqs N} \Lambda(s_i),\\
C_{\hat{\gamma}} &= \max_{1\leqs i\leqs N} C_0(s_i),\\
\mcal{C}_{a_{\hat{\gamma}}} &= \vset{\varphi \in L^1 (\lambda) : \varphi \geqs 0, \;\: \varphi \not\equiv 0, \;\: |\varphi|_s \leqs \left(\min_{1\leqs i\leqs N}a(s_i)\right)\left(\min_{1\leqs i\leqs N}\mbb{E}[\varphi|\mscr{Q}_{n_0(s_i)}(s_i)]\right)}.
\end{align*}

\section{Nonstationary open piecewise-smooth expanding systems in higher dimension}
\label{s:app}

We apply our general framework for conditional memory loss to a class of piecewise-$C^{1+\alpha}$ expanding maps introduced by Saussol~\cite{SaussolQuasi2000} and then studied in the nonstationary context by Gupta, Ott, and T\"{o}r\"{o}k~\cite{GuptaOttTorok2013}.
These maps exhibit rich dynamics, yet are simple enough to allow us to implement our general framework without overburdening the exposition.

When holes are not present, Saussol~\cite{SaussolQuasi2000} has shown that each map in this class admits finitely many ergodic absolutely continuous invariant probability measures (ACIPs) with quasi-H\"{o}lder densities.
Gupta, Ott, and T\"{o}r\"{o}k~\cite{GuptaOttTorok2013} establish statistical memory loss at an exponential rate for nonstationary compositions of such maps.

Here, we establish conditional memory loss at an exponential rate when holes are present.
We allow both the maps and the holes to evolve in time.
The results in~\cite{GuptaOttTorok2013, SaussolQuasi2000} depend on balancing expansion and complexity.
We show here how to obtain such a balance in the nonstationary open context.

\subsection{Phase space and conventions}

For the sake of clarity, our phase space will be the $N$-dimensional torus $\mbb{T}^{N} = \mbb{R}^{N} / \mbb{Z}^{N}$, although our results extend naturally to compact Riemannian manifolds.
Recall that $\lambda$ denotes the Lebesgue measure on $\mbb{T}^{N}$.
For a set $S \subset \mbb{T}^{N}$ and $\varepsilon > 0$, define
\begin{equation*}
B_{\varepsilon} (S) = \bigcup_{x \in S} B_{\varepsilon} (x),
\end{equation*}
where $B_{\varepsilon} (x)$ (and $B(x,\varepsilon)$) denote the open ball of radius $\varepsilon$ centered at $x$.

We fix a H\"{o}lder exponent $0 < \alpha \leqs \Lip$ throughout this section.
All maps are assumed to be piecewise-$C^{1+\alpha}$, unless we explicitly state other regularity properties.
For a map with $C^{1+\Lip}$ regularity, we use $\alpha = 1$ in the computations.
For $\hat{f} : U \subset \mbb{R}^{k} \to \mbb{R}^{l}$, we use the $C^{1 + \alpha}$ and $C^{2}$ norms
\begin{equation*}
\norm{\hat{f}}_{C^{1+\alpha}} = \norm{\hat{f}}_{C^{0}} + \norm{D\hat{f}}_{C^{\alpha}}, \qquad \norm{\hat{f}}_{C^{2}} = \norm{\hat{f}}_{C^{0}} + \norm{D\hat{f}}_{C^{0}} + \norm{D^{2} \hat{f}}_{C^{0}}.
\end{equation*}

\subsection{Domains for piecewise-continuous maps}
\label{ss:partitions}

The maps we consider are piecewise-continuous on {\itshape open partitions} of $\mbb{T}^{N}$; that is, on families of open, pairwise-disjoint sets that cover $\mbb{T}^{N}$ up to sets of Lebesgue measure zero.
We shall refer to the open partitions on which our maps are piecewise-continuous as continuity partitions.
When forming nonstationary compositions of the maps we consider, we allow both the maps and the continuity partitions to vary in time.
We must therefore topologize the space of continuity partitions when we topologize the space of maps we consider.
A notion of continuity partition complexity will play an important role in our conditional memory loss results, as is to be expected:
Complexity plays a vital role in the ergodic theory of piecewise-hyperbolic systems.

\begin{definition}[Partitions]
\label{def:parts}
Let $\mscr{A} = \set{U_{i}: 1 \leqs i \leqs L}$ be a finite open partition of $\mbb{T}^{N}$.  
For $K > 0$, we say that $\mscr{A} \in \mcal{P}(K)$ if each $U_{i}$ has boundaries bounded piecewise in $C^{2}$ by $K$.  
More precisely:
\begin{enumerate}[leftmargin=*, labelindent=\parindent, label=(\alph*), ref=\alph*]
\item 
$\lambda(\mbb{T}^{N} \setminus{\bigcup_{i} U_{i}}) = 0$.
\item 
For each $i$, there are finitely many compact $C^{2}$ embedded codimension-one submanifolds $\set{\Gamma_{ij}}_{j}$ of $\mbb{T}^{N}$ such that $\partial U_{i}$ is contained in the union $\bigcup_{j} \Gamma_{ij}$. \label{cond:a}
\item 
The $C^{2}$-norm of each $\Gamma_{ij}$ is strictly less than $K$:
For each $\Gamma_{ij}$, there exist finitely many $C^{2}$ charts $\Phi_{l;ij} : B_{N} \subset \mbb{R}^{N} \to W_{l;ij}\subset \mbb{T}^{N},$ where $B_{N}$ is the unit ball of $\mbb{R}^{N},$ such that each $\Phi_{l;ij}$ and $\Phi_{l;ij}^{-1}$ has $C^{2}$-norm less than $K$, and $\Gamma_{ij} \subset \bigcup_{l}\Phi_{l;ij}(B_{N} \cap (\mbb{R}^{N-1} \oplus \set{0})).$ \label{cond:b}
\end{enumerate}
Let $\mcal{P} = \bigcup_{K > 0} \mcal{P}(K).$
\end{definition}

\begin{definition}[Partition complexity]
\label{d:partition-complexity}
For $\mscr{A}\in \mcal{P}$, define
\begin{equation*}
\kappa(\mscr{A}) = \sup_{x \in \mbb{T}^{N}} \#\vset{\Gamma_{ij} : x \in \Gamma_{ij}}.
\end{equation*}
One can view $\kappa(\mscr{A})$ as a measure of complexity for the partition $\mscr{A}.$
See~\cite{Cowieson_2000,Cowieson_2002} for a related notion.
\end{definition}

\begin{definition}[Nearby partitions]
\label{def:nearbyparts}
We say that two open partitions $\mscr{A} = \set{U_{i}}$ and $\widetilde{\mscr{A}} = \{\widetilde{U}_{i}\}$ in $\mcal{P}(K)$ are $\delta$-close if the following hold.
\begin{enumerate}[leftmargin=*, labelindent=\parindent, label=(\alph*), ref=\alph*]
\item 
The families $\mscr{A}$ and $\widetilde{\mscr{A}}$ have the same number of elements, and there also exists a correspondence between the bounding submanifolds $\Gamma_{ij}$ and $\widetilde{\Gamma}_{ij}$.
\item 
For each $i$, $\mrm{d}_{\Hausd} (U_{i}, \widetilde{U}_{i}) < \delta$, where $\mrm{d}_{\text{Hausd}}$ denotes the Hausdorff distance defined by
\begin{equation*}
\mrm{d}_{\Hausd} (A,B) = \max\vset{\sup\vset{\text{dist}(x,B) : x \in A}, \; \sup\vset{\text{dist}(y,A) : y \in B}}.
\end{equation*}
\item 
For each $i$ and $j$, the bounding submanifolds $\Gamma_{ij}$ and $\widetilde{\Gamma}_{ij}$ are less than $\delta$ apart in Hausdorff distance.
\end{enumerate}
\end{definition}

This notion of perturbation defines a topology on $\mcal{P}(K).$  
With respect to this topology, the complexity $\kappa$ is upper-semicontinuous.  
More precisely:

\begin{lemma}
Given $\mscr{A} \in \mcal{P}(K)$, there exists $\delta > 0$ such that $\kappa(\widetilde{\mscr{A}}) \leqs \kappa(\mscr{A})$ for each $\widetilde{\mscr{A}} \in \mcal{P}(K)$ that is $\delta$-close to $\mscr{A}.$
\end{lemma}

\begin{proof}
Suppose $x_{n} \in \mbb{T}^{N}$ is a point in $k$ boundary components of $\mscr{A}_{n}$ with $\mscr{A}_{n}$ being $\delta_{n}$-close to $\mscr{A}$, along a sequence $\delta_{n} \to 0$.
By passing to a subsequential limit if necessary, we have that $x_{n} \to x$ and $x_{n} \in \bigcap_{(i,j) \in J} \Gamma_{ij}^{n}$ for a set $J$ with $k$ elements, with each boundary component $\Gamma_{ij}^{n}$ of $\mscr{A}_{n}$ being $\delta_{n}$-close to the corresponding boundary component $\Gamma_{ij}$ of $\mscr{A}.$  As $n\to \infty$, this implies that $x\in \bigcap_{J} \Gamma_{ij}.$
\end{proof}

\begin{definition}[Piecewise-continuous maps]
For $\mscr{A} \in \mcal{P},$ let $C(\mscr{A})$ denote the set of maps $\hat{f} : \mbb{T}^{N} \to \mbb{T}^{N}$ that are continuous on each $U \in \mscr{A}.$  
Let $C(\mcal{P})$ denote the union of $C(\mscr{A})$ over all $\mscr{A} \in \mcal{P}$.
\end{definition}

\subsection{Two classes of piecewise-smooth expanding maps}

We now introduce the piecewise-$C^{1+\alpha}$ expanding maps we analyze, and do so in two stages.
The class $\mscr{M}$ consists of maps with a suitable balance between expansion and complexity.
Even for iterates of a single map, not requiring such a balance can lead to maps with exotic ergodic properties, such as the absence of ACIPs~\cite{buzzi2001}.
A Lasota-Yorke inequality (see~\cite{SaussolQuasi2000, GuptaOttTorok2013}) holds for a class $\mscr{M}^{*} \subset \mscr{M}$ of maps with enhanced regularity properties.
We will build nonstationary open systems using maps in $\mscr{M}^{*}$.

\begin{notation}
Let $\xi_{N} = \pi^{N/2} / (N / 2)!$ denote the volume of the unit ball in $\mbb{R}^{N}$, where $(N / 2)! = \Gamma(N / 2 + 1).$
\end{notation}

\begin{definition}[The class $\mscr{M}$]
Let $0 < s < 1$, $K > 0$, and $\kappa > 0$.
Assume the complexity-expansion balance
\begin{equation}\label{eq:complexity_v_expansion}
s^{\alpha} + \left(\frac{4s\kappa}{1 - s}\right)\left(\frac{\xi_{N - 1}}{\xi_{N}}\right)< 1.
\end{equation}
Let $\mscr{M}(s,K,\kappa)$ denote the set of piecewise-$C^{1 + \alpha}$ maps $\hat{h} : \mbb{T}^{N} \to \mbb{T}^{N}$ that satisfy the following properties for some associated partition $\mscr{A} = \mscr{A} (\hat{h}) = \set{U_{i} : 1 \leqs i \leqs L} \in \mcal{P} (K)$.
\begin{enumerate}[leftmargin=*, labelindent=\parindent, label=(\alph*), ref=\alph*]
\item 
$\hat{h} \in C(\mscr{A})$. 
\item 
For each $i$, $\hat{h}|_{U_{i}}$ is injective with a differentiable inverse and $\norm{D[(\hat{h}|_{U_{i}})^{-1}]} < s$.
(Backward contraction)
\item 
For each $i,$ both $\hat{h}|_{U_{i}}$ and all its partial derivatives extend to continuous functions on the closure of $U_{i}$.
\item 
For each $i$, $\norm{\hat{h}|_{U_{i}}}_{C^{1 + \alpha}} < K$. 
\item 
$\kappa( \mscr{A} ) \leqs \kappa $.
\end{enumerate}
Let $\mscr{M} = \bigcup_{(s, K, \kappa)} \mscr{M} (s,K,\kappa)$.
\end{definition}

\begin{definition}[The class $\mscr{M}^{*}$]
For $\varepsilon_{0} > 0$, let $\mscr{M}^{*}(s,K,\kappa, \vep_{0})$ denote the set of maps $\hat{h} \in \mscr{M}(s,K,\kappa)$ that can be extended to neighborhoods of the original sets $U_{i}$ as follows: 
For each $i,$ there is an open set $V_{i}\supset \overline{U}_{i}$ and an extension $\hat{h}_{(i)} : V_{i} \to \mbb{T}^{N}$ of $\hat{h}|_{U_{i}}$ such that 
\begin{enumerate}[leftmargin=*, labelindent=\parindent, label=(\alph*), ref=\alph*]
\item 
$\hat{h}_{(i)}(V_{i}) \supset \overline{B_{\vep_{0}}(\hat{h} (U_{i}))}$,
\item 
$\hat{h}_{(i)}$ is a $C^{1}$ diffeomorphism from $V_{i}$ to its image,
\item 
$\big\| D [\hat{h}_{(i)}^{-1}] \big\| < s$ on $\hat{h}_{(i)}(V_{i})$
(Backward contraction of extensions),
\item 
$\norm{\hat{h}_{(i)}}_{1 + \alpha} < K$ on $V_{i}$.
\end{enumerate}
Let $\mscr{M}^{*} = \bigcup_{(s, K, \kappa, \vep_{0})} \mscr{M}^{*} (s, K, \kappa, \vep_{0})$.
\end{definition}

\subsection{Topologizing the space of piecewise-smooth expanding maps}

For a fixed map $\hat{g} \in \mscr{M} (s, K, \kappa)$, we now describe perturbations of $\hat{g}$ that will generate a basis for a topology on $\mscr{M} (s, K, \kappa)$.   
This topology is similar to the topology used in~\cite{demersliverani2008, gouezelliverani2006}.

\begin{definition}
Let $\hat{h}, \hat{g} \in \mscr{M} (s,K,\kappa)$.
We say that $\hat{h}$ is a $\delta$-perturbation of $\hat{g}$ and write $\hat{h} \in \mscr{N}(\hat{g},\delta) = \mscr{N}(\hat{g},\delta; s, K, \kappa)$ if the following hold.
\begin{enumerate}[leftmargin=*, labelindent=\parindent, label=(\alph*), ref=\alph*]
\item 
$\hat{h} \in C(\mscr{A} (\hat{h}))$ with $\mscr{A} (\hat{h}) = \set{\widetilde{U}_{i}} \in \mcal{P} (K)$ $\delta$-close to $\mscr{A} (\hat{g}) = \set{U_{i}} \in \mcal{P} (K)$, as described in Definition~\ref{def:nearbyparts}.
\item 
Outside a $\delta$-neighborhood of the boundaries, the maps are $\delta$-close in $C^{1 + \alpha}$:
\begin{equation*}
\big\| \hat{h}|_{W_{i}} - \hat{g}|_{W_{i}} \big\|_{C^{1 + \alpha}} < \delta 
\end{equation*}
for each $i$, where
\begin{equation*}
W_{i} = \vset{ x \in U_{i} \cap \widetilde{U}_{i} : \dist (x,U_{i}^{c}) > \delta, \; \dist (x,\widetilde{U}_{i}^{c}) > \delta}.
\end{equation*}
\end{enumerate}
\end{definition}

\noindent
As $\hat{g}$ and $\delta$ vary, the sets $\mscr{N}(\hat{g},\delta; s, K, \kappa)$ form a basis for a topology on $\mscr{M}(s,K,\kappa)$.

\subsection{Holes}
\label{ss:holes}

In order to formulate conditional memory loss results for nonstationary open systems built from maps in $\mscr{M}^{*}$, we allow holes with small Lebesgue measure and controlled geometric complexity.
We use the former to verify~\pref{A:3} - stability of mixing.
We use the latter to establish a Lasota-Yorke/Doeblin-Fortet inequality.

\begin{definition}[$\mscr{H}$ - Admissible holes]
Let $H$ be an open subset of $\mbb{T}^{N}$ and $K > 0.$  
We say that $H \in \mscr{H}(K)$ if the following hold.
\begin{enumerate}[leftmargin=*, labelindent=\parindent, label=(\alph*), ref=\alph*]
\item 
There exist finitely many compact $C^{2}$ embedded codimension-one submanifolds $\vset{\Gamma_{j}}$ of $\mbb{T}^{N}$ such that $\partial H$ is contained in the union $\bigcup_{j} \Gamma_{j}$. 
\item 
The $C^{2}$-norm of each $\Gamma_{j}$ is strictly less than $K$:
For each $\Gamma_{j}$, there exist finitely many $C^{2}$ charts $\Phi_{l,j} : B_{N} \subset \mbb{R}^{N} \to W_{l,j}\subset \mbb{T}^{N},$ where $B_{N}$ is the unit ball of $\mbb{R}^{N},$ such that each $\Phi_{l,j}$ and $\Phi_{l,j}^{-1}$ has $C^{2}$-norm less than $K$, and $\Gamma_{j} \subset \bigcup_{l}\Phi_{l,j}(B_{N} \cap (\mbb{R}^{N-1} \oplus \vset{0})).$
\end{enumerate}
We define the set of admissible holes by $\mscr{H} = \bigcup_{K > 0}\mscr{H}(K).$
\end{definition}

\subsection{Densities}
\label{ss:quasiHolder}
We work with spaces of quasi-H\"older densities.
See e.g.~\cite{GuptaOttTorok2013, SaussolQuasi2000} for details about the properties of these quasi-H\"{o}lder spaces.   

For $\vp \in L^{1}(\lambda)$ and a Borel set $S \subset \mbb{T}^{N}$, define the {\itshape oscillation} of $\vp$ on $S$ by
\begin{equation*}
\osc(\vp, S) = \Esup (\vp,S) - \Einf (\vp,S), 
\end{equation*}
where $\Esup$ and $\Einf$ denote the essential supremum and essential infimum, respectively.
Given $\ve_{0} > 0,$ define the seminorm
\begin{equation}
\label{eq:osc-seminorm}
\abs{\vp}_{\alpha, \ve_{0}} = \sup_{0 < \ve \leqs \ve_{0}} \ve^{-\alpha} \int_{\mbb{T}^{N}} \osc(\vp, B_{\ve}(x))\;\mrm{d}\lambda(x).
\end{equation}
The seminorms $\abs{\vp}_{\alpha, \ve_{0}}$ are equivalent for different values of $\ve_{0}$.
This follows from the fact that for $0 < \ve_{1} < \ve_{2}$, there are finitely many vectors $\mbi{v}_{i}$ such that $B_{\ve_{2}}(x) \subset \bigcup_{i} B_{\ve_{1}}(x + \mbi{v}_{i})$ for all $x.$  Define
\begin{equation*}
\OSC_{\alpha} = \vset{\vp \in L^{1}(\lambda) : \abs{\vp}_{\alpha, \ve_{0}} < \infty}.
\end{equation*}
This space does not depend on $\ve_{0}$ and contains the $\alpha$-H\"older functions.  Define the norm $\vnorm{\cdot}_{\alpha,\ve_{0}}$ on $\OSC_{\alpha}$ by
\begin{equation}\label{eq:osc-norm}
\vnorm{\vp}_{\alpha,\ve_{0}} = \vnorm{\vp}_{L^{1}(\lambda)} + \abs{\vp}_{\alpha,\ve_{0}}.
\end{equation}
Equipped with this norm, $\OSC_{\alpha}$ is a Banach space and the unit ball of $\OSC_{\alpha}$ is precompact in $L^{1}(\lambda).$

For our conditional memory loss results, we work with densities in $\OSC_{\alpha}$:
\begin{equation*}
\mcal{D} = \vset{ \varphi \in \OSC_{\alpha} : \varphi \geqs 0, \; \: \norm{\varphi}_{L^{1} (\lambda)} = 1 }.
\end{equation*}

\subsection{Complexity for nonstationary open systems}

Gupta, Ott, and T\"{o}r\"{o}k prove that certain nonstationary systems built from maps in $\mscr{M}^{*}$ (no holes) exhibit statistical memory loss at an exponential rate~\cite{GuptaOttTorok2013}.
Crucial to these results is the balance between expansion and complexity~\eqref{eq:complexity_v_expansion} enjoyed by maps in $\mscr{M}^{*}$.
The introduction of holes into a nonstationary system can upset the expansion-complexity balance.
To handle this, we define a {\itshape complexity sequence} for nonstationary open systems.
Our definition leverages the fact that our holes and continuity partition elements possess comparable boundary regularity.

Let $(\hat{f}_{i})$ be a sequence of maps in $\mscr{M}$ and let $(H_{j})$ be a sequence of holes in $\mscr{H}$.  
We define a complexity sequence for the corresponding nonstationary open system $((F_{m}),(H_{j}))$ using dynamical partitions. 

\begin{definition}[Dynamical partitions for $\widehat{F}_{m}$ and $F_{m}$]

The dynamical partition for $\widehat{F}_{m} = \hat{f}_{m} \circ \cdots \circ \hat{f}_{1}$ is the join
\begin{equation*}
\mscr{A} (\widehat{F}_{m}) = \mscr{A} (\hat{f}_{1}) \vee \bigvee_{k=2}^{m} (\hat{f}_{k} \circ \cdots \circ \hat{f}_{2})^{-1} (\mscr{A} (\hat{f}_{1}))
\end{equation*}
for $m \geqs 2$.
We define the dynamical partition for $F_{m}$ $(m \geqs 1)$ by intersecting with the survivor set $S_{m}$:
\begin{equation*}
\mscr{A} (F_{m}) = \mscr{A} (\widehat{F}_{m}) \cap S_{m} \defas \vset{ V \cap S_{m} : V \in \mscr{A} (\widehat{F}_{m}) }.
\end{equation*}
\end{definition}

\begin{definition}[Complexity sequence]
\label{def:complexity-sequence}
Define the {\itshape complexity sequence} $(\kappa_{m})_{m=1}^{\infty}$ for the nonstationary open system $((F_{m}),(H_{j}))$ by $\kappa_{m} = \kappa (\mscr{A} (F_{m}))$, where partition complexity $\kappa (\cdot)$ is as in Definition~\ref{d:partition-complexity}.
\end{definition}

As we shall see, an inequality of Lasota-Yorke/Doeblin-Fortet type follows from suitable control of the complexity sequence.

\subsection{Verification of~\pref{A:1}--\pref{A:3}}
\label{ss:verification}

We show that the setting of Section~\ref{s:app} fits into the abstract framework of Section~\ref{s:framework1}.

\subsection*{\pref{A:1}}

We opt to work with systems built from maps in $\mscr{M}^{*}$ (rather than $\mscr{M}$) because the extension properties enjoyed by maps in $\mscr{M}^{*}$ facilitate the development of a Lasota-Yorke inequality.
Let $0 < s < 1$ and $\vep_{0} > 0$.
Let $(\hat{f}_{i})_{i=1}^{\infty}$ be a sequence of maps in $\mscr{M}^{*} (s, \vep_{0})$.
(Here $\mscr{M}^{*} (s, \vep_{0})$ denotes the union of $\mscr{M}^{*} (s, K, \kappa, \vep_{0})$ over $K > 0$ and $\kappa > 0$.)
Let $(H_{j})_{j=1}^{\infty}$ be a sequence of holes in $\mscr{H}$.

To establish a Lasota-Yorke inequality for the nonstationary open system $((F_{m}),(H_{j}))$, we look for times at which a balance between complexity and expansion emerges.

\begin{proposition}[Inequality of Lasota-Yorke/Doeblin-Fortet type]
\label{lydf-app}
Let $0 < s < 1$ and $\vep_{0} > 0$.
Let $(\hat{f}_{i})_{i=1}^{\infty}$ be a sequence of maps in $\mscr{M}^{*} (s, \vep_{0})$ and $(H_{j})_{j=1}^{\infty}$ be a sequence of holes in $\mscr{H}$.
Suppose $T_{1} \in \mbb{N}$ is a time for which the complexity-expansion balance
\begin{equation}
\label{ineq:balance}
s^{T_{1} \alpha} + \left( \frac{4 s^{T_{1}} \kappa_{T_{1}}}{1 - s^{T_{1}}} \right) \left( \frac{\xi_{N-1}}{\xi_{N}} \right) < 1
\end{equation}
holds for the nonstationary open system $((F_{m}), (H_{j}))$.
Assume $F_{T_{1}} \in \mscr{M}^{*} (s^{T_{1}}, \vep_{0})$; let $K_{T_{1}} > 0$ be such that $F_{T_{1}} \in \mscr{M}^{*} (s^{T_{1}}, K_{T_{1}}, \kappa_{T_{1}}, \vep_{0})$.
(Here we extend the definition of $\mscr{M}^{*}$ to include maps that are defined only on survivor sets.)
Then there exist positive constants $\vep_{\LY} = \vep_{\LY} (s^{T_{1}}, K_{T_{1}}, \kappa_{T_{1}}, \vep_{0}, N) \leqs \vep_{0}$, $\hat{\theta}_{\LY} = \hat{\theta}_{\LY} (s^{T_{1}}, K_{T_{1}}, \kappa_{T_{1}}, \vep_{\LY}) < 1$, and $\widehat{C}_{\LY} = \widehat{C}_{\LY} (s^{T_{1}}, K_{T_{1}}, \kappa_{T_{1}}, \vep_{\LY})$ such that
\begin{equation}
\label{e:LY-GOT}
\abs{\mcal{L}_{F_{T_{1}}} (\varphi)}_{\alpha , \vep_{\LY}} \leqs \hat{\theta}_{\LY} \abs{\varphi}_{\alpha , \vep_{\LY}} + \widehat{C}_{\LY} \norm{\varphi}_{L^{1} (\lambda)}
\end{equation}
for every $\varphi \in L^{1} (\lambda)$.
\end{proposition}

The proof of this inequality for systems without holes appears in~\cite{SaussolQuasi2000, GuptaOttTorok2013}.
The introduction of holes does not substantially alter the proof, so we omit the details here.
Inequality~\eqref{e:LY-DF} may be obtained by iterating~\eqref{e:LY-GOT}, for instance.

\subsection*{\pref{A:2}}
Assumption~\pref{A:2} is used only to prove Lemma~\ref{L:control}. 
We opt here to directly establish Lemma~\ref{L:control} for the oscillation seminorm in order to demonstrate how this seminorm behaves in analytical estimates.

\begin{proof}[Proof of Lemma~\ref{L:control} for the oscillation seminorm]
We make adjustments to the lower bounds on $\varphi (z)$ and $\mbb{E} [\mcal{L}_{F} (\varphi) | \mscr{Q}_{n_{0}}] (x)$.
First, for almost every $z \in Q' \cap F^{-1} (Q(x))$ we have
\begin{align*}
\varphi (z) &\geqs \left( \frac{1}{\lambda (Q')} \int_{Q'} \varphi \, \mrm{d} \lambda \right) - \osc (\varphi, Q')
\\
&\geqs \frac{1}{\lambda (Q')} \left( \int_{Q'} \varphi \, \mrm{d} \lambda - \int_{Q'} \osc (\varphi, B(y, \diam (\mscr{Q}_{n_{0}}))) \, \mrm{d} \lambda (y) \right).
\end{align*}
Here $\diam (\mscr{Q}_{n_{0}}) = \sup_{Q' \in \mscr{Q}_{n_{0}}} \diam (Q')$ and $\diam (\cdot)$ refers to metric diameter, as opposed to the measure-theoretic partition diameter $\diam_{\lambda}$ that appears in Definition~\ref{defn:mix}.
Second, writing $\widehat{E}$ for $\mbb{E} [\mcal{L}_{F} (\varphi) | \mscr{Q}_{n_{0}}]$, we have
\begin{align*}
\widehat{E} (x) &\geqs \frac{1}{\lambda (Q(x))} \sum_{Q' \in \mscr{Q}_{n_{0}}} \int_{Q' \cap F^{-1} (Q(x))} \frac{1}{\lambda (Q')} \left( \int_{Q'} \varphi \, \mrm{d} \lambda - \int_{Q'} \osc (\varphi, B(y, \diam (\mscr{Q}_{n_{0}}))) \, \mrm{d} \lambda (y) \right) \mrm{d} \lambda (z)
\\
&= \sum_{Q' \in \mscr{Q}_{n_{0}}} \frac{\lambda (Q' \cap F^{-1} (Q(x)))}{\lambda (Q(x)) \lambda (Q')} \left( \int_{Q'} \varphi \, \mrm{d} \lambda - \int_{Q'} \osc (\varphi, B(y, \diam (\mscr{Q}_{n_{0}}))) \, \mrm{d} \lambda (y) \right)
\\
&\geqs \zeta_{1} \int_{\mbb{T}^{N}} \varphi \, \mrm{d} \lambda - \zeta_{2} |\varphi|_{\alpha, \vep_{\LY}} (\diam (\mscr{Q}_{n_{0}}))^{\alpha}
\\
&\geqs (\zeta_{1} - \zeta_{2} a (\diam (\mscr{Q}_{n_{0}}))^{\alpha}) \int_{\mbb{T}^{N}} \varphi \, \mrm{d} \lambda
.
\end{align*}
Here we assume that $\diam (\mscr{Q}_{n_{0}}) \leqs \vep_{\LY}$.
We work with $\vep_{\LY}$ instead of $\vep_{0}$ because in the setting of Section~\ref{ss:verification}, the Lasota-Yorke inequality holds for $| \cdot |_{\alpha, \vep_{\LY}}$.
The proof of the upper bound in Lemma~\ref{L:control} proceeds analogously.
We have therefore established Lemma~\ref{L:control} for the oscillation seminorm with $M = (\diam (\mscr{Q}_{n_{0}}))^{1 - \alpha}$.
\end{proof}

\subsection*{\pref{A:3}}

Stability of mixing is automatic in the setting of Section~\ref{ss:verification}.

\subsection{Putting the pieces together}

\begin{MainTheorem}
In the setting of Section~\ref{ss:verification}, Theorem~\ref{thrm:local} holds with the following adjustments.
\begin{itemize}
\item
The space $\mscr{M}^{*} (s, \vep_{0})$ replaces $\mscr{M}$.
\item
Hypothesis~\pref{A:2} is not needed.
\item
The space $\mscr{E} (\zeta_{1}, \zeta_{2})$ is now defined in terms of metric diameter $\diam (\cdot)$ instead of measure-theoretic diameter $\diam_{\lambda} (\cdot)$.
\item
We assume control of the complexity-expansion balance for the nonstationary open systems.
In particular, we assume that~\eqref{e:LY-GOT} can be iterated.
See Proposition~\ref{lydf-app}. 
\item
One must adjust the parameter selection procedure to account for the following (see~\cite{GuptaOttTorok2013} for guidance).
\begin{itemize}
\item
We require $\diam (\mscr{Q}_{n_{0}}) \leqs \vep_{\LY}$.
\item
We use $M = (\diam (\mscr{Q}_{n_{0}}))^{1 - \alpha}$.
\end{itemize}
\end{itemize}
\end{MainTheorem}

Theorem~\ref{thrm:global} holds in the setting of Section~\ref{ss:verification} as well, provided one makes analogous adjustments.

\section{Future directions}
\label{s:future}

We conclude by discussing some directions for future research.
We formulate open-ended queries, as opposed to precise conjectures.

\subsection*{Complexity control for nonstationary open systems via transversality}
A key component in the statement and proof of Proposition~\ref{lydf-app} is the existence of a time for which~\eqref{ineq:balance} holds, which in turn relies heavily on the behavior of the complexity sequence $(\kappa_{m})$ as time progresses.
The existence of such a time is not guaranteed {\itshape a priori}.
The dynamical partition $\mscr{A}(F_m)$ may be quite complicated and $\kappa_{m}$ may grow quickly as a consequence.
Further, piecewise-smooth expanding systems may exhibit pathological behavior.
For instance, Tsujii~\cite{Tsujii_2000} constructs, for each $1 \leqs r < \infty$, a piecewise-$C^r$ expanding map $\hat{f}: D \to D$ on an open rectangle $D \subset \mbb{R}^2$ such that the dynamical partition for $\hat{f}$ has infinitely many connected components.
Tsujii also establishes the existence of an open set $B \subset D$ such that the diameter of $\hat{f}^n(B)$ converges to $0$ as $n\to\infty,$ showing that iterates of $\hat{f}$ `cut up' $B$ into many small pieces.
However, one expects that such pathologies will not occur if the sequence of maps $(\hat{f}_{i})$ possesses transversality properties and if the holes are placed in `general position'.

Cowieson shows in~\cite{Cowieson_2002} that for a certain class of nonstationary closed systems $(\widehat{F}_{m})$, a certain transversality assumption implies that the complexity $\hat{\kappa}_m \defas \kappa(\mscr{A} (\widehat{F}_{m}))$ remains bounded above by a constant that depends only on the initial partition $\mscr{A}(\hat{f}_1)$ of $\mbb{T}^{N}$.
Our first future direction suggests generalizing the Cowieson ideas to the nonstationary open context.

\begin{direction}
  In the context of Section~\ref{s:app}, formulate a notion of transversality for nonstationary open systems that implies meaningful control of the complexity sequence.
  (Ideally, transversality would imply that the complexity sequence remains bounded.)
  Show that most (in an appropriate sense) nonstationary open systems satisfy this notion of transversality.
\end{direction}

The interplay between complexity, hole geometry, and hole placement is subtle.
The presence of holes complicates the complexity control problem for a number of reasons.
First, the domain of the system at time $m$ (that is, the time-$m$ survivor set) is a function of $m$.
Second, holes can potentially increase or decrease complexity.
On one hand, hole boundaries introduce additional complexity.
On the other hand, an area of high complexity for a nonstationary closed system could fall into a hole of an associated nonstationary open system at time $m$, dramatically reducing $\kappa_{m}$.

\subsection*{General nonstationary open systems}

We conclude with general future directions.

\begin{direction}
  In the context of nonstationary open systems, analyze the limit in which hole size tends to zero.
\end{direction}

\begin{direction}
  Develop an operator-theoretic framework (using anisotropic Banach spaces, perhaps) that facilitates the analysis of nonstationary open systems constructed using maps that possess both expanding and contracting behavior.
\end{direction}

\begin{acknowledgments}
This work has been partially supported by NSF grant DMS-1413437 (WO).
\end{acknowledgments}

\bibliographystyle{siam}
\bibliography{edge-arxiv-v5}

\begin{thebibliography}{10}

\bibitem{Aimino-polynomial-2015}
{\sc R.~Aimino, H.~Hu, M.~Nicol, A.~T\"or\"ok, and S.~Vaienti}, {\em Polynomial
  loss of memory for maps of the interval with a neutral fixed point}, Discrete
  Contin. Dyn. Syst., 35 (2015), pp.~793--806.

\bibitem{Arnold-RDS-1998}
{\sc L.~Arnold}, {\em Random dynamical systems}, Springer Monographs in
  Mathematics, Springer-Verlag, Berlin, 1998.

\bibitem{Baladi-Demers-Liverani-2018}
{\sc V.~Baladi, M.~F. Demers, and C.~Liverani}, {\em Exponential decay of
  correlations for finite horizon {S}inai billiard flows}, Invent. Math., 211
  (2018), pp.~39--177.

\bibitem{Baladi-Tsujii-2007}
{\sc V.~Baladi and M.~Tsujii}, {\em Anisotropic {H}\"older and {S}obolev spaces
  for hyperbolic diffeomorphisms}, Ann. Inst. Fourier (Grenoble), 57 (2007),
  pp.~127--154.

\bibitem{Birkhoff-Lattice-1979}
{\sc G.~Birkhoff}, {\em Lattice theory}, vol.~25 of American Mathematical
  Society Colloquium Publications, American Mathematical Society, Providence,
  R.I., third~ed., 1979.

\bibitem{Blanchard-2018}
{\sc A.~E. Blanchard, C.~Liao, and T.~Lu}, {\em Circuit-host coupling induces
  multifaceted behavioral modulations of a gene switch}, Biophysical Journal,
  114 (2018), pp.~737--746.

\bibitem{Bressaud-Liverani-2002}
{\sc X.~Bressaud and C.~Liverani}, {\em Anosov diffeomorphisms and coupling},
  Ergodic Theory Dynam. Systems, 22 (2002), pp.~129--152.

\bibitem{Bunimovich-Yurchenko-2011}
{\sc L.~A. Bunimovich and A.~Yurchenko}, {\em Where to place a hole to achieve
  a maximal escape rate}, Israel J. Math., 182 (2011), pp.~229--252.

\bibitem{buzzi2001}
{\sc J.~Buzzi}, {\em No or infinitely many a.c.i.p. for piecewise expanding
  {$C^r$} maps in higher dimensions}, Comm. Math. Phys., 222 (2001),
  pp.~495--501.

\bibitem{Chernov-coupling-2006}
{\sc N.~Chernov}, {\em Advanced statistical properties of dispersing
  billiards}, J. Stat. Phys., 122 (2006), pp.~1061--1094.

\bibitem{Chernov-Dolgopyat-2009}
{\sc N.~Chernov and D.~Dolgopyat}, {\em Brownian {B}rownian motion. {I}}, Mem.
  Amer. Math. Soc., 198 (2009), pp.~viii+193.

\bibitem{Cowieson_2000}
{\sc W.~J. Cowieson}, {\em Stochastic stability for piecewise expanding maps in
  {${\bf R}^d$}}, Nonlinearity, 13 (2000), pp.~1745--1760.

\bibitem{Cowieson_2002}
\leavevmode\vrule height 2pt depth -1.6pt width 23pt, {\em Absolutely
  continuous invariant measures for most piecewise smooth expanding maps},
  Ergodic Theory Dynam. Systems, 22 (2002), pp.~1061--1078.

\bibitem{Demers-Wright-LSY-2010}
{\sc M.~Demers, P.~Wright, and L.-S. Young}, {\em Escape rates and physically
  relevant measures for billiards with small holes}, Comm. Math. Phys., 294
  (2010), pp.~353--388.

\bibitem{Demers-logistic-2005}
{\sc M.~F. Demers}, {\em Markov extensions and conditionally invariant measures
  for certain logistic maps with small holes}, Ergodic Theory Dynam. Systems,
  25 (2005), pp.~1139--1171.

\bibitem{Demers-billiards-2014}
\leavevmode\vrule height 2pt depth -1.6pt width 23pt, {\em Dispersing billiards
  with small holes}, in Ergodic theory, open dynamics, and coherent structures,
  vol.~70 of Springer Proc. Math. Stat., Springer, New York, 2014,
  pp.~137--170.

\bibitem{demersliverani2008}
{\sc M.~F. Demers and C.~Liverani}, {\em Stability of statistical properties in
  two-dimensional piecewise hyperbolic maps}, Trans. Amer. Math. Soc., 360
  (2008), pp.~4777--4814.

\bibitem{Demers-Young-open-2006}
{\sc M.~F. Demers and L.-S. Young}, {\em Escape rates and conditionally
  invariant measures}, Nonlinearity, 19 (2006), pp.~377--397.

\bibitem{Demers-Zhang-2011}
{\sc M.~F. Demers and H.-K. Zhang}, {\em Spectral analysis of the transfer
  operator for the {L}orentz gas}, J. Mod. Dyn., 5 (2011), pp.~665--709.

\bibitem{Dettmann-Rahman-2014}
{\sc C.~P. Dettmann and M.~R. Rahman}, {\em Survival probability for open
  spherical billiards}, Chaos, 24 (2014), pp.~043130, 15.

\bibitem{Dobbs-Stenlund-2017}
{\sc N.~Dobbs and M.~Stenlund}, {\em Quasistatic dynamical systems}, Ergodic
  Theory Dynam. Systems, 37 (2017), pp.~2556--2596.

\bibitem{gouezelliverani2006}
{\sc S.~Gou\"ezel and C.~Liverani}, {\em Banach spaces adapted to {A}nosov
  systems}, Ergodic Theory Dynam. Systems, 26 (2006), pp.~189--217.

\bibitem{GuptaOttTorok2013}
{\sc C.~Gupta, W.~Ott, and A.~T\"or\"ok}, {\em Memory loss for time-dependent
  piecewise expanding systems in higher dimension}, Math. Res. Lett., 20
  (2013), pp.~141--161.

\bibitem{Kunita-flows-1990}
{\sc H.~Kunita}, {\em Stochastic flows and stochastic differential equations},
  vol.~24 of Cambridge Studies in Advanced Mathematics, Cambridge University
  Press, Cambridge, 1990.

\bibitem{LeJan-1985}
{\sc Y.~Le~Jan}, {\em On isotropic {B}rownian motions}, Z. Wahrsch. Verw.
  Gebiete, 70 (1985), pp.~609--620.

\bibitem{Ledrappier-Young-1988}
{\sc F.~Ledrappier and L.-S. Young}, {\em Entropy formula for random
  transformations}, Probab. Theory Related Fields, 80 (1988), pp.~217--240.

\bibitem{Liverani-cones-1995}
{\sc C.~Liverani}, {\em Decay of correlations}, Ann. of Math. (2), 142 (1995),
  pp.~239--301.

\bibitem{Liverani_Deschamps_2003}
{\sc C.~Liverani and V.~Maume-Deschamps}, {\em Lasota-{Y}orke maps with holes:
  conditionally invariant probability measures and invariant probability
  measures on the survivor set}, Ann. Inst. H. Poincar\'e Probab. Statist., 39
  (2003), pp.~385--412.

\bibitem{LSV_1998}
{\sc C.~Liverani, B.~Saussol, and S.~Vaienti}, {\em Conformal measure and decay
  of correlation for covering weighted systems}, Ergodic Theory Dynam. Systems,
  18 (1998), pp.~1399--1420.

\bibitem{Masmoudi-Young-2002}
{\sc N.~Masmoudi and L.-S. Young}, {\em Ergodic theory of infinite dimensional
  systems with applications to dissipative parabolic {PDE}s}, Comm. Math.
  Phys., 227 (2002), pp.~461--481.

\bibitem{Mattingly-1999}
{\sc J.~C. Mattingly}, {\em Ergodicity of {$2$}{D} {N}avier-{S}tokes equations
  with random forcing and large viscosity}, Comm. Math. Phys., 206 (1999),
  pp.~273--288.

\bibitem{Melbourne-Terhesiu-2012}
{\sc I.~Melbourne and D.~Terhesiu}, {\em Operator renewal theory and mixing
  rates for dynamical systems with infinite measure}, Invent. Math., 189
  (2012), pp.~61--110.

\bibitem{Mohapatra_Ott_2014}
{\sc A.~Mohapatra and W.~Ott}, {\em Memory loss for nonequilibrium open
  dynamical systems}, Discrete Contin. Dyn. Syst., 34 (2014), pp.~3747--3759.

\bibitem{Ott-Stenlund-Young-2009}
{\sc W.~Ott, M.~Stenlund, and L.-S. Young}, {\em Memory loss for time-dependent
  dynamical systems}, Math. Res. Lett., 16 (2009), pp.~463--475.

\bibitem{Pianigiani-Yorke-1979}
{\sc G.~Pianigiani and J.~A. Yorke}, {\em Expanding maps on sets which are
  almost invariant. {D}ecay and chaos}, Trans. Amer. Math. Soc., 252 (1979),
  pp.~351--366.

\bibitem{SaussolQuasi2000}
{\sc B.~Saussol}, {\em Absolutely continuous invariant measures for
  multidimensional expanding maps}, Israel J. Math., 116 (2000), pp.~223--248.

\bibitem{Stenlund-Anosov-2011}
{\sc M.~Stenlund}, {\em Non-stationary compositions of {A}nosov
  diffeomorphisms}, Nonlinearity, 24 (2011), pp.~2991--3018.

\bibitem{Stenlund-LSY-Zhang-2013}
{\sc M.~Stenlund, L.-S. Young, and H.~Zhang}, {\em Dispersing billiards with
  moving scatterers}, Comm. Math. Phys., 322 (2013), pp.~909--955.

\bibitem{Tsujii_2000}
{\sc M.~Tsujii}, {\em Piecewise expanding maps on the plane with singular
  ergodic properties}, Ergodic Theory Dynam. Systems, 20 (2000),
  pp.~1851--1857.

\bibitem{Veliz-Cuba-2016}
{\sc A.~Veliz-Cuba, C.~Gupta, M.~R. Bennett, K.~Josi\'{c}, and W.~Ott}, {\em
  Effects of cell cycle noise on excitable gene circuits}, Physical Biology, 13
  (2016).

\bibitem{Young-towers-1998}
{\sc L.-S. Young}, {\em Statistical properties of dynamical systems with some
  hyperbolicity}, Ann. of Math. (2), 147 (1998), pp.~585--650.

\bibitem{Young-coupling-1999}
\leavevmode\vrule height 2pt depth -1.6pt width 23pt, {\em Recurrence times and
  rates of mixing}, Israel J. Math., 110 (1999), pp.~153--188.

\end{thebibliography}

\end{document}